\documentclass[12 pt]{amsart}
\pdfoutput=1 % makes arXiv use pdflatex
\usepackage{amssymb, amsmath, amsfonts, amsthm, graphics}
\usepackage[hmargin=1 in, vmargin = 1 in]{geometry} 
\usepackage{tikz}
\usepackage{tikz-cd} 
\usepackage{dynkin-diagrams}
\usepackage[colorlinks=true]{hyperref}
\usepackage{cleveref}
\usepackage{mathtools}
\usepackage{float}
\usepackage[all]{xy}
\usepackage{colortbl}

\newcommand\Tstrut{\rule{0pt}{4ex}}         % = `top' strut
\newcommand\Bstrut{\rule[-3ex]{0pt}{0pt}}   % = `bottom' strut

\usepackage{color}

\newcommand{\margincolor}{red}      
\definecolor{darkgreen}{rgb}{0,0.7,0}

\newcounter{margincounter}
\newcommand{\marginnum}{
\ifnum\value{margincounter}<10
\textcolor{\margincolor}{\begin{picture}(0,0)\put(2.2,2.4){\circle{9}}\end{picture}\footnotesize\arabic{margincounter}}
\else\ifnum\value{margincounter}<100
\textcolor{\margincolor}{\begin{picture}(0,0)\put(4.256,2.5){\circle{11}}\end{picture}\footnotesize\arabic{margincounter}}
\else
\textcolor{\margincolor}{\begin{picture}(0,0)\put(6.8,2.5){\circle{14}}\end{picture}\footnotesize\arabic{margincounter}}
\fi\fi
}

\newcommand{\newword}[1]{\textbf{\emph{#1}}}

\newcommand{\bigjoin}{\bigvee}
\newcommand{\bigmeet}{\bigwedge}

\newcommand{\NN}{\mathbb{N}}

\newcommand{\RR}{\mathbb{R}}

\newcommand{\ZZ}{\mathbb{Z}}

\newcommand{\cC}{\mathcal{C}}

\newcommand{\cX}{\mathcal{X}}

\newcommand{\tA}{\widetilde{A}}
\newcommand{\tB}{\widetilde{B}}
\newcommand{\tC}{\widetilde{C}}
\newcommand{\tD}{\widetilde{D}}
\newcommand{\tE}{\widetilde{E}}
\newcommand{\tF}{\widetilde{F}}
\newcommand{\tG}{\widetilde{G}}

\newcommand{\uPhi}{\widehat{\Phi}}
\newcommand{\uPi}{\widehat{\Pi}}

\newcommand{\uV}{\widehat{V}}
\newcommand{\ugamma}{\widehat{\gamma}}
\newcommand{\ubeta}{\widehat{\beta}}
\newcommand{\ualpha}{\widehat{\alpha}}

\newcommand{\utheta}{\widehat{\theta}}

\newcommand{\Span}{\text{Span}}

\newtheorem{Theorem}{Theorem}[section]

\newtheorem{theorem}[Theorem]{Theorem}
\newtheorem*{Atheorem}{Theorem A}
\newtheorem*{Btheorem}{Theorem B}
\newtheorem*{Ctheorem}{Theorem C}
\newtheorem*{Dtheorem}{Theorem D}

\newtheorem{lemma}[Theorem]{Lemma}

\newtheorem{prop}[Theorem]{Proposition}
\newtheorem{conjecture}[Theorem]{Conjecture}

\theoremstyle{definition}

\newtheorem{remark}[Theorem]{Remark}

\newcommand{\Spanp}{\Span_+}
\newcommand{\Vd}{V^*}
\newcommand{\pos}{\Phi^+}


\title{Affine extended weak order is a lattice}
\author{Grant T. Barkley}
\address{Department of Mathematics, Harvard University, Cambridge, MA, USA}
\email{gbarkley@math.harvard.edu}
\author{David E Speyer}
\address{Department of Mathematics, University of Michigan, Ann Arbor, MI, USA}
\email{speyer@umich.edu}

\begin{document}

\begin{abstract}
	Coxeter groups are equipped with a partial order known as the weak order, such that $u \leq v$ if the inversions of $u$ are a subset of the inversions of $v$.
	In finite Coxeter groups, weak order is a complete lattice, but in infinite Coxeter groups it is only a meet semi-lattice.
	Motivated by questions in Kazhdan--Lusztig theory, Matthew Dyer introduced a larger poset, now known as extended weak order, which contains the weak order as an order ideal and coincides with it for finite Coxeter groups. The extended weak order is the containment order on certain sets of positive roots: those which satisfy a geometric condition making them ``biclosed''. The finite biclosed sets are precisely the inversion sets of Coxeter group elements.
	Generalizing the result for finite Coxeter groups, Dyer conjectured that the extended weak order is always a complete lattice, even for infinite Coxeter groups.
	
	In this paper, we prove Dyer's conjecture for Coxeter groups of affine type.  
	To do so, we introduce the notion of a clean arrangement, which is a hyperplane arrangement where the regions are in bijection with biclosed sets. We show that root poset order ideals in a finite or rank 3 untwisted affine root system are clean. We set up a general framework for reducing Dyer's conjecture to checking cleanliness of certain subarrangements. We conjecture this framework can be used to prove Dyer's conjecture for all Coxeter groups.
\end{abstract}

\maketitle

%\tableofcontents

\section{Introduction}
The weak order is a partial order on a Coxeter group $W$. When $W$ is the symmetric group $S_n$, weak order coincides with inclusion order on the inversions of the permutations. Using root systems, one can define inversions for elements of an arbitrary Coxeter group and similarly characterize the weak order. From this perspective, weak order is tied heavily to the geometry of root systems and Coxeter arrangements. (A Coxeter arrangement is the arrangement of hyperplanes dual to a root system.) For instance, in a finite Coxeter group, the Hasse diagram of the weak order has the same underlying graph as the $1$-skeleton of the permutahedron for that group. One can use the geometry of Coxeter arrangements to understand the weak order and vice-versa (as accomplished thoroughly in \cite{Reading2016Ch10,Reading2016Ch9}). For instance, the fact that the regions of a finite Coxeter arrangement are all simplicial implies that the weak order of a finite Coxeter group admits meets and joins \cite{BEZ}. In other words, finite weak order is a \emph{lattice}. But for infinite Coxeter groups this is no longer the case. It was shown by Bj\"orner \cite{Bjorner1984} that the weak order is a meet-semilattice but never a lattice for such groups.
In this paper we continue a line of research, motivated by the work of Matthew Dyer, attempting to embed the weak order on infinite Coxeter groups into some larger complete lattice.

We discuss what could be the elements of such a larger lattice. 
To any Coxeter group $W$, there is an associated real root system $\Phi$ in a real vector space $V$; we write $\pos$ for the positive roots. 
There is then a dual hyperplane arrangement $\bigcup_{\beta \in \pos} \beta^{\perp}$ in the dual vector space $\Vd$.
To any region $R$ in the hyperplane arrangement complement, we can associate the set of positive roots $\beta$ such that $\langle -, \beta \rangle$ is negative on $R$.
A set of roots of this form is called \newword{separable}. The finite separable sets are precisely the sets of inversions of elements of $W$.
Once $W$ is infinite, the separable sets almost never form a lattice. We will therefore discuss larger classes of subsets of $\pos$ which include the separable sets and might form a lattice.

We can more generally consider \newword{weakly separable} sets of roots. A subset $C$ of $\pos$ is called \newword{weakly separable} if, for any finite collections $\beta_1$, $\beta_2$, $\dots$, $\beta_i$ of roots in $C$ and  $\gamma_1$, $\gamma_2$, $\dots$, $\gamma_j$ in $\pos \setminus C$, there is some $\theta \in \Vd$ with $\langle \theta, -\rangle$ 
negative on the $\beta$'s and positive on the $\gamma$'s. In rank $3$ Coxeter groups, weakly separable sets of roots form a lattice~\cite[Section 2.4]{Berlin2013} but, even for rank $4$ affine Coxeter groups, the weakly separable sets already do not form a lattice.

Dyer introduced the notion of \newword{biclosed sets}, first studied by Papi \cite{Papi1994}, which are more general than weakly separable sets.
A subset $B$ of $\pos$ is \newword{biclosed} if, for any two positive roots $\alpha$, $\beta$, and any root $\gamma$ which is a non-negative linear combination of $\alpha$ and $\beta$, whenever $\alpha$ and $\beta\in B$ then $\gamma\in B$, and whenever $\alpha$ and $\beta\not\in B$ then $\gamma\not\in B$. 
Equivalently, a biclosed set is one whose restriction to any 2-dimensional root subsystem is the inversion set of a region of the corresponding rank 2 Coxeter arrangement. The poset of all biclosed sets under inclusion is called the \newword{extended weak order} of the Coxeter group. 
The finite biclosed sets are exactly the inversions of elements of $W$ \cite[Proposition 4.2]{Blessenohl2005} (see also \cite{Papi1994,Dyer2019}). 
As a result, when $W$ is a finite Coxeter group, the extended weak order coincides with the weak order on $W$; in particular, it is a lattice.

In affine Coxeter groups of rank $3$, the biclosed sets are precisely the weakly separable sets \cite{Barkley2022}.
For higher rank Coxeter groups, biclosed sets are more general. For the affine Coxeter groups, the authors give a full classification of the biclosed sets in~\cite{Barkley2022}; this was also done in unpublished work of Dyer~\cite{Dyerpreprint}.
The current work is inspired by Dyer's conjecture:

\begin{conjecture}[\cite{Dyer2019,Dyer1994}]\label{conj:Dyer}
	The extended weak order of any Coxeter group is a complete lattice.
\end{conjecture}

The first accomplishment of this paper is a proof of Conjecture~\ref{conj:Dyer} for affine Coxeter groups.
\begin{Atheorem}\hypertarget{thm:introlattice}
	The extended weak order of an affine Coxeter group is a complete lattice.
\end{Atheorem}
Until now this theorem was known only for the rank 3 affine groups \cite{Wang2018} and for types $\tA$ and $\tC$ \cite{Barkley2022}.
The key to the proof of Theorem~\hyperlink{thm:introlattice}{A} is to develop a method to reduce the general problem to the rank $3$ case, which we now explain.
Let $V$ be a finite dimensional real vector space and let $X$ be a finite subset of $V$ lying in an open halfspace.
We define separable subsets of $X$ and biclosed subsets of $X$ similarly to how we did for $\pos$ above. (See Section~\ref{sec:setsofroots} for details.)
We define $X$ to be \newword{clean} if every biclosed subset of $X$ is separable; we call the dual hyperplane arrangement $\bigcup_{\beta \in X} \beta^{\perp}$ to be a \newword{clean arrangement}. 
Remarks~\ref{rem:cleanlit} and \ref{rem:cleanrootposet} discuss previous work on clean arrangements and related notions.

Let $\Phi \subset V$ be a root system. 
We will define an ordering $\beta_1$, $\beta_2$, $\beta_3$, \dots of the positive roots $\pos$ to be \newword{suitable} if, for every initial section $\{ \beta_1, \beta_2, \ldots, \beta_N \}$ and every rank $3$ subsystem $R$, the intersection $\{ \beta_1, \beta_2, \ldots, \beta_N \} \cap R$ is clean.
(See~\Cref{sec:subsystems} for the definition of a rank $3$ subsystem.)
% Our main technical tool (see \Cref{thm:BiclosedLattice} for the full, stronger, statement) says that, if $\pos$ has a suitable ordering, then biclosed sets of $\pos$ form a complete lattice.
Assume $\pos$ has a suitable ordering. Our main technical results (see \Cref{thm:KeyExtensionLemma} and its variants in \Cref{sec:extend}) are about extending subsets of initial sections to biclosed sets in $\pos$. 
Using \Cref{thm:KeyExtensionLemma}, we deduce that if $\pos$ has a suitable ordering, then biclosed sets of $\pos$ form a complete lattice (\Cref{thm:BiclosedLattice}).
We also show that, if $B$ is biclosed in  $\{ \beta_1, \beta_2, \ldots, \beta_N \}$, then there is a unique minimal biclosed set $\overline{B}$ in $\pos$ with $B = \overline{B} \cap \{ \beta_1, \beta_2, \ldots, \beta_N \}$. 
From this, many statements about extended weak order reduce to statements about the finite poset of biclosed subsets of $\{\beta_1,\ldots,\beta_N\}$.
% As a result, suitable order ideals give lattice quotients...
% This lets us reduce statements about extended weak order to statements about the finite poset of biclosed sets in $\{\beta_1,\ldots,\beta_N\}$. 

In \Cref{sec:prelim}, we define the various properties a set of vectors can have and define the types of root system we will consider. In \Cref{sec:extend}, we prove the main structure theorem on biclosed sets in the presence of suitable orders. In \Cref{sec:biclosedlattice}, we show how this structure theorem implies the lattice property. The key input to the results of those sections is the existence of a suitable ordering; in \Cref{sec:EasySets} we give basic examples of suitable orders.
In \Cref{sec:SuitableStart} we discuss preliminary steps to proving that an order is suitable.

We then turn to the task of proving that affine root systems have suitable orders.
Let $\Phi$ be an untwisted affine root system.
We will show that, if we take the (crystallographic) root poset on $\pos$, and take any total order refining it, this is a suitable order.
The restriction of the root order to any rank $3$ subsystem is a refinement of the root order on that subsystem.
Thus, concretely, what we ultimately show is:
\begin{Btheorem}\hypertarget{thm:introclean}
	Let $\Phi$ be a finite crystallographic root system or a rank 3 untwisted affine root system.
	Let $I$ be an order ideal in the root poset on $\pos$. Then $I$ is clean.
\end{Btheorem}
We believe even the finite-type statement here is new. 
We thus also deduce
\begin{Ctheorem} \hypertarget{thm:introextend}
	Let $\Phi$ be a finite crystallographic root system or a rank 3 untwisted affine root system and let $I$ be an order ideal in the root poset on $\pos$. Let $B$ be biclosed in $I$. 
	Then there is a unique minimal biclosed set $\overline{B}$ in $\pos$ such that $\overline{B} \cap I = B$.
\end{Ctheorem}

Our work is also applicable to another conjecture of Dyer, which is referred to as ``Conjecture A''. It is the analog of the fact that maximal chains in the weak order on a finite Coxeter group (equivalently, reduced expressions for the longest element $w_0$) correspond to \emph{reflection orders} on $\pos$. A reflection order is a total ordering of $\pos$ so that every initial section of the order is biclosed. In any Coxeter group, the collection of initial sections of a fixed reflection order is a maximal chain in the extended weak order. Conjecture A asserts the converse.
\begin{conjecture}[Conjecture A \cite{Dyer2019,Dyerpreprint}]\label{conj:ConjA}
	Any maximal chain in the extended weak order is the set of initial sections of a unique total ordering on the positive roots.
\end{conjecture}
Dyer has already proven Conjecture A for affine Coxeter groups in an unreleased work~\cite{Dyerpreprint}. We give a new proof of this fact for affine Coxeter groups, by showing (in \Cref{thm:ConjectureA}) that any root system with a suitable ordering satisfies the conjecture.

\begin{Dtheorem} \hypertarget{thm:introConjD}
	Let $\Phi$ be a finite crystallographic root system or an untwisted affine root system.
	Any maximal chain in the extended weak order on $\Phi$ is the set of initial sections of a unique total ordering on the positive roots.
\end{Dtheorem}

We remark that we have reduced Dyer's Conjectures \ref{conj:Dyer} and \ref{conj:ConjA}, which are the two ``main conjectures'' on biclosed sets, to the problem of finding a suitable order on a root system. We expect that this can always be done.
\begin{conjecture}\label{conj:suitable}
	Any root system has a suitable order.
\end{conjecture}

We note that \Cref{conj:suitable} implies the following conjecture. Its possibility was raised in \cite{McConvilleComm} and \cite[Remark 2.15]{Hohlweg2016}.

\begin{conjecture}\label{conj:cleanrank3}
	If $\Phi$ is a rank 3 root system, then $\pos$ is clean.
\end{conjecture}

We now discuss how we have chosen to organize this paper. 
As described above, Sections~\ref{sec:prelim} through~\ref{sec:SuitableStart} introduce the general properties of suitable orderings, but do not prove that particular root systems have suitable orders.
In order to construct suitable orders, one must prove Theorem~\hyperlink{thm:introclean}{B} in types $A_3$, $B_3$, $C_3$, $\tA_2$, $\tC_2$ and $\tG_2$.
It is possible to verify all six types by extensive case checking.
The authors have carried out this task, but it requires many cases.

Instead, we  carry out the direct proof of Theorem~\hyperlink{thm:introclean}{B} only in types $A_3$ and $\tA_2$.
Once we have done this, we will be able to deduce   Theorems~\hyperlink{thm:introlattice}{A} and~\hyperlink{thm:introextend}{C} in the simply laced types: $A_n$, $\tA_n$,  $D_n$, $\tD_n$, and $E_n$ and $\tE_n$ for $n=6$, $7$, $8$.

We then use the technique of ``folding'' to transfer results along $A_5 \to C_3$, $D_4 \to B_3$, $\tA_3 \to \tC_2$ and $\tD_4 \to \tG_2$. We use these foldings to deduce Theorem~\hyperlink{thm:introclean}{B} in all crystallographic finite and affine types.
This then implies Theorem~\hyperlink{thm:introlattice}{A},~\hyperlink{thm:introextend}{C}, and~\hyperlink{thm:introConjD}{D} in all of these cases.

In \Cref{sec:FiniteSuitable,sec:A2TildeSuitable} we prove Theorem~\hyperlink{thm:introclean}{B} in types $A_3$ and $\tA_2$. 
In \Cref{sec:Folding}, we carry out the folding argument to deduce Theorem~\hyperlink{thm:introclean}{B} in types $B_3$, $C_3$, $\tC_2$ and $\tG_2$.
In \Cref{sec:FoldedTypes}, we conclude the proofs of our main results.
In \Cref{sec:TwistedAffine,sec:TypeH}, we explain how our results change in twisted affine root systems and in non-crystallographic root systems.

GTB was supported by NSF grants DMS-2152991, DMS-1854512, and DMS-1600223. DES was supported by NSF grants DMS-2246570, DMS-1600223, DMS-1854225 and DMS-1855135. The authors would also like to thank Matthew Dyer and Thomas McConville for helpful communications.

\section{Preliminaries}\label{sec:prelim}

\subsection{Notions of closure in sets of vectors}\label{sec:setsofroots}
Let $V$ be a finite dimensional real vector space and let $\Vd$ be the dual space. For a subset $A$ of $V$, we write $\Spanp(A)$ for the set of nonnegative linear combinations of vectors in A.

Let $X$ be a subset of $V$, and $B$ a subset of $X$. We make the following (standard) definitions:

\begin{itemize}
	\item $B$ is \newword{closed in $X$} if, whenever $\alpha$ and $\beta \in B$, and $\gamma\in \Spanp(\alpha,\beta)\cap X$, then $\gamma\in B$. We say that $B$ is \newword{coclosed in $X$} if $X\setminus B$ is closed in $X$.
	
	\item $B$ is \newword{convex in $X$} if 
	%whenever $\alpha_1,\alpha_2,\ldots,\alpha_k\in B$ and $\gamma\in \Spanp(\alpha_1,\alpha_2,\ldots,\alpha_k)\cap X$, then $\gamma\in B$. 
	%Equivalently, 
	$\Spanp(B)\cap X = B$. 
	We say that $B$ is \newword{coconvex in $X$} if $X\setminus B$ is convex in $X$. 
	
	\item $B$ is \newword{biclosed in $X$} if $B$ is closed and coclosed in $X$.
	
	\item $B$ is \newword{biconvex in $X$} if $B$ is convex and coconvex in $X$.
	
	\item $B$ is \newword{weakly separable in $X$} if $\Spanp(B)\cap \Spanp(X\setminus B)=\{0\}$. 
	%, for any $\alpha_1$, $\alpha_2$, \dots, $\alpha_k \in B$ and $\beta_1$, $\beta_2$, \dots, $\beta_{\ell} \in X \setminus B$, we have $\Spanp(\alpha_1, \alpha_2, \ldots, \alpha_k) \cap \Spanp(\beta_1, \beta_2, \ldots, \beta_{\ell}) = \{ 0 \}$.
	
	\item $B$ is \newword{separable in $X$} if there is a dual vector $\theta\in \Vd$ such that $B= \{\alpha\in X : \langle \theta, \alpha\rangle < 0\}$ and $X\setminus B = \{ \alpha\in X : \langle \theta, \alpha \rangle > 0 \}.$ 
\end{itemize}
We have the immediate implications shown in Figure~\ref{fig:Implications}.

\begin{figure}[h]
	\[ \xymatrix{
		& \text{separable} \ar@{=>}[d]  & \\
		& \text{weakly separable} \ar@{=>}@<-1ex>[d] & \\
		\text{convex} \ar@{=>}[d]  &\text{biconvex}  \ar@{=>}[r] \ar@{=>}[l]  \ar@{=>}@<-1ex>[d] \ar@2{-->}@<-1ex>[u]  & \text{coconvex} \ar@{=>}[d]  \\
		\text{closed} &\text{biclosed} \ar@2{-->}@<-1ex>[u]  \ar@{=>}[r] \ar@{=>}[l]  & \text{coclosed} \\
	} \]
	\caption{Implications between different notions of convexity are shown with solid arrows. If $X$ is clean, then the dashed implications also hold.} \label{fig:Implications}
\end{figure}

If $X$ is finite, then Farkas' lemma (see, e.g., \cite{StoerJosef2012}) 
states that weakly separable implies separable. For infinite $X$, the hyperplane separation theorem \cite{StoerJosef2012} %Thm 3.3.4
states that for any weakly separable $B$, there is some nonzero functional $\theta\in \Vd$ such that $\langle \theta, \alpha \rangle \leq 0$ for all $\alpha\in B$ and such that $\langle \theta, \alpha\rangle \geq 0$ for all $\alpha\in X\setminus B$. See \cite{Berlin2013,Hohlweg2016} for more comparisons between these definitions and examples distinguishing them (note that their ``separable'' is our ``weakly separable''). We say that $X$ is \newword{clean} if every biclosed set in $X$ is also weakly separable. The implications of cleanliness are shown in \Cref{fig:Implications}.  See \Cref{sec:clean} for more discussion.

\subsection{Coxeter groups and root systems}
A \newword{Coxeter system} is a pair $(W,S)$ consisting of a group $W$ and a set $S = \{s_1,\ldots,s_n\}\subseteq W$ of order-2 elements such that $S$ generates $W$ and the relations between elements of $S$ give $W$ a presentation of the form
\[ W \cong \langle s\in S \mid (s_is_j)^{m_{ij}} = 1 \rangle, \]
where $(m_{ij})_{i,j=1}^n$ is a symmetric matrix such that $m_{ii} = 1$ for all $i$ and $m_{ij} \in \{2,3,4,\ldots, \infty \}$. When $m_{ij}=\infty$, this means that there is no relation of the form $(s_is_j)^m=1$ appearing in the presentation. The matrix $(m_{ij})$ is called the \newword{Coxeter matrix} and determines the system up to isomorphism. The elements of $S$ are called \newword{simple generators} and the cardinality of $S$ is called the \newword{rank} of the Coxeter system. When $W$ is part of an understood Coxeter system $(W,S)$, we say that $W$ is a \newword{Coxeter group} and suppress $S$ from the notation. We say $W$ is \newword{reducible} if $S$ can be partitioned into nonempty sets $S_1,S_2$ such that $s$ and $s'$ commute when $s\in S_1$ and $s'\in S_2$, and otherwise we say $W$ is \newword{irreducible}.

Coxeter groups naturally arise as groups of reflections acting on a vector space. There is a very general notion of a ``root system'' governing this correspondence; we will focus on a special case here. Let $V$ be a real vector space equipped with a symmetric bilinear form $(-,-)$. 
Given a vector $\alpha\in V$ such that $(\alpha,\alpha)\neq 0$, we can define the \newword{reflection} over $\alpha$ to be the linear map $t_\alpha : V\to V$ defined by
\[ t_{\alpha}(\beta) = \beta - 2\frac{(\alpha,\beta)}{(\alpha,\alpha)}\alpha. \]
A (symmetrizable, crystallographic, real, reduced) \newword{root system} in $V$ is a subset $\Phi\subseteq V$ satisfying the following properties:
\begin{itemize}
	\item For all $\alpha\in \Phi$ we have $(\alpha,\alpha)>0$ and $t_\alpha\Phi = \Phi$.
	\item For all $\alpha\in \Phi$ we have $\RR\alpha\cap \Phi = \{\pm \alpha\}$.  %For all $\alpha,\beta\in \Phi$ we have $2\frac{(\alpha,\beta)}{(\alpha,\alpha)}\in\ZZ$.
	\item There exists a set $\Pi\subseteq \Phi$ satisfying:
	\begin{itemize}
		\item For all $\alpha\in \Phi$, either $\alpha \in \Spanp(\Pi)$ or $-\alpha\in \Spanp(\Pi)$ but not both.
		\item For all $\alpha\in \Pi$, we have $\alpha \not\in \Spanp(\Pi\setminus\{\alpha\})$.
		\item $\Phi \subseteq \Span_{\ZZ}\Pi$.
	\end{itemize}
\end{itemize}
A set $\Pi$ as above is called a \newword{base} for the root system $\Phi$. The root systems we will consider come with a chosen base $\Pi=\{\alpha_1,\ldots, \alpha_n \}$. The elements of $\Pi$ are called \newword{fundamental roots}. The size of $\Pi$ is called the (abstract) \newword{rank} of $\Phi$. When needed, we call the dimension of the span of $\Phi$ the \newword{linear rank} of $\Phi$. If the elements of $\Pi$ are linearly independent (equivalently, rank equals linear rank) then we say that $\Phi$ is \newword{geometrically embedded}. We write $\pos \coloneqq \Phi\cap \Spanp\Pi$ for the set of roots in the nonnegative span of $\Pi$; its elements are called the \newword{positive roots} of $\Phi$. Similarly we can define $\Phi^-$, the \newword{negative roots} of $\Phi$. Then the defining properties of $\Pi$ imply that $\Phi = \pos\sqcup \Phi^-$. We say two root systems $\Phi_1, \Phi_2$ in vector spaces $V_1,V_2$ are \newword{abstractly isomorphic} if there is a bijection between $\Phi_1$ and $\Phi_2$ preserving the bases, the pairing $(-,-)$, and the action of reflections.

Given a root system $\Phi$ with base $\Pi = \{\alpha_1,\ldots,\alpha_n\}$, we define the \newword{Weyl group} of $\Phi$ to be the group $W$ which is the subgroup of $\mathrm{GL}(V)$ generated by $\{t_\alpha \mid \alpha\in \Phi \}$. A fundamental consequence of the root system axioms is that $W$ is in fact generated by $S\coloneqq \{t_\alpha\mid \alpha\in \Pi\}$, and the pair $(W,S)$ is a Coxeter system. Hence we will freely refer to the Weyl group of $\Phi$ as a Coxeter group. We write $s_i \coloneqq t_{\alpha_i}$. Any root system gives rise to a unique Coxeter group in this way, but there may be multiple isomorphism classes of root systems associated to a given Coxeter group.

Associated to any root system with base $\Pi=\{\alpha_1,\ldots,\alpha_n\}$ is the matrix $(a_{ij})_{i,j=1}^n$, where $a_{ij} = 2\frac{(\alpha_i,\alpha_j)}{(\alpha_i,\alpha_i)}$. This is the \newword{Cartan matrix} of $\Phi$. Our definition of a root system implies that each $a_{ij}\in\ZZ$. The Cartan matrix does not determine the embedding $\Phi\hookrightarrow V$ since, for instance, we can take $V$ to have arbitrarily large dimension. It also does not determine the lengths of roots in $\Phi$. However, associated to each Cartan matrix and consistent choice of root length, there is a unique triple $(\Phi,\Pi,V)$ up to isomorphism such that $\Pi$ is a basis of $V$. 
This is the \newword{geometric realization} of $\Phi$.

We say that a root system $\Phi$ with base $\Pi$ is \newword{reducible} if one of the following equivalent conditions is satisfied:
\begin{itemize}
	\item There is a partition $\Phi = \Phi_1 \sqcup \Phi_2$ into nonempty parts such that if $\alpha\in\Phi_1$ and $\beta\in\Phi_2$ then $(\alpha,\beta)=0$.
	\item There is a partition $\Phi = \Pi_1 \sqcup \Pi_2$ into nonempty parts such that if $\alpha\in\Pi_1$ and $\beta\in\Pi_2$ then $(\alpha,\beta)=0$.
	\item The Weyl group of $W$ is reducible.
\end{itemize}
Otherwise, we say that $\Phi$ is \newword{irreducible}. Irreducible root systems are determined by their Cartan matrix up to an overall normalizing constant. An arbitrary root system decomposes into irreducible components.
% Furthermore, if $\Phi$ is irreducible, then any embedding $\Phi\hookrightarrow V$ is a geometric embedding. 
% Proof that irreducible implies rank equals linear rank: If $\Phi$ is any root system, then any $\alpha\in \Phi$ is in the positive rational span of some linearly independent set $S\subseteq \Pi$. Then necessarily $\alpha$ is in the subsystem generated by $S$ by induction on height (where height is computed in some geometric embedding of $\Phi$).

\subsection{Finite and affine root systems}

The finite irreducible root systems $\Phi$ are classified by \newword{Dynkin diagrams} (see \Cref{fig:finitedynkin}). 
\begin{figure}
	\centering
	\begin{tabular}{|c|c|c|c|}
		\hline
		Name & Diagram & Name & Diagram \\
		\hline 
		$A_n$\rule{0pt}{2.6ex} & \dynkin[scale=1.5,labels={1,2,n-1,n},edge length=.5cm]{A}{} & $E_6$ & \dynkin[scale=1.5,label]{E}{6} \\
		\hline
		$B_n$\rule{0pt}{2.6ex} & \dynkin[scale=1.5,labels={1,2,n-2,n-1,n},edge length=.5cm]{B}{} & $E_7$ & \dynkin[scale=1.5,label]{E}{7} \\ %label
		\hline
		$C_n$\rule{0pt}{2.6ex} & \dynkin[scale=1.5,labels={1,2,n-2,n-1,n},edge length=.5cm]{C}{} & $E_8$ & \dynkin[scale=1.5,label]{E}{8} \\ %label
		\hline
		$D_n$ & \dynkin[scale=1.5,labels={1,2,n-3,n-2,n-1,n},label directions={,,,right,,},edge length = .4cm]{D}{} 	& $F_4$\rule{0pt}{2.6ex} & \dynkin[scale=1.5,label]{F}{4} \\ %label
		\hline
		& & $G_2$\rule{0pt}{2.6ex} & \dynkin[scale=1.5,label]{G}{2} \\ %label
		\hline
		%$H_3$ & \dynkin[scale=1.5,Coxeter]{H}{3} \\
		%\hline
		%$H_4$ & \dynkin[scale=1.5,Coxeter]{H}{4} \\
		%\hline
		%$I_2(m)$ & \dynkin[scale=1.5,Coxeter,gonality=m]{I}{2} \\
		%\hline
	\end{tabular}
	\caption{The Dynkin diagrams associated to finite irreducible root systems.}\label{fig:finitedynkin}
\end{figure}

%\begin{figure}
%	\centering
%	\begin{tabular}{|c|c|}
	%		\hline
	%		Name & Diagram \\
	%		\hline
	%		$A_n$\rule{0pt}{2.6ex} & \dynkin[scale=1.5,edge length=.5cm]{A}{} \\ %,labels={1,2,n-2,n-1}
	%		\hline
	%		$B_n$\rule{0pt}{2.6ex} & \dynkin[scale=1.5,edge length=.5cm]{B}{} \\ %,labels={1,2,n-2,n-1,n}
	%		\hline
	%		$C_n$\rule{0pt}{2.6ex} & \dynkin[scale=1.5,edge length=.5cm]{C}{} \\ %labels={1,2,n-2,n-1,n}
	%		\hline
	%		$D_n$ & \dynkin[scale=1.5,label directions={,,,right,,},edge length = .4cm]{D}{} \\ %labels={1,2,n-3,n-2,n-1,n}
	%		\hline
	%	\end{tabular}
%	\hspace{3em}
%	\begin{tabular}{|c|c|}
	%		\hline
	%		Name & Diagram \\
	%		\hline 
	%		$E_6$ & \dynkin[scale=1.5]%,label,extended]
	%		{E}{6} \\
	%		\hline
	%		$E_7$ & \dynkin[scale=1.5]{E}{7} \\ %label
	%		\hline
	%		$E_8$ & \dynkin[scale=1.5]{E}{8} \\ %label
	%		\hline
	%		$F_4$\rule{0pt}{2.6ex} & \dynkin[scale=1.5]{F}{4} \\ %label
	%		\hline
	%		$G_2$\rule{0pt}{2.6ex} & \dynkin[scale=1.5]{G}{2} \\ %label
	%		\hline
	%		%$H_3$ & \dynkin[scale=1.5,Coxeter]{H}{3} \\
	%		%\hline
	%		%$H_4$ & \dynkin[scale=1.5,Coxeter]{H}{4} \\
	%		%\hline
	%		%$I_2(m)$ & \dynkin[scale=1.5,Coxeter,gonality=m]{I}{2} \\
	%		%\hline
	%	\end{tabular}
%	\caption{The Dynkin diagrams associated to finite irreducible root systems.}\label{fig:finitedynkin}
%\end{figure}
%
The number of nodes in a Dynkin diagram is the rank of its associated root system $\Phi$. The node labeled $i$ in the Dynkin diagram corresponds to the fundamental root $\alpha_i$ of $\Phi$. A single edge between nodes $i$ and $j$ indicates that the Cartan matrix entries $a_{ij}$ and $a_{ji}$ are both $-1$. A double edge with an arrow pointing from $i$ to $j$ indicates that $a_{ij} = -1$ and $a_{ji} = -2$. A triple edge with an arrow pointing from $i$ to $j$ indicates that $a_{ij} = -1$ and $a_{ji}=-3$. (In particular, arrows always point from longer roots to shorter roots.) If there is no edge between $i$ and $j$, then $a_{ij}=a_{ji}=0$.  We will normalize our finite root systems in the usual way, so that the shortest roots $\alpha$ of each irreducible component all satisfy $(\alpha,\alpha)=2$. A root system $\Phi$ in a vector space $V$ with bilinear form $(-,-)$ is a finite root system if and only if the restriction of $(-,-)$ to the span of $\Phi$ is positive definite. Furthermore, any finite root system is always geometrically embedded. 

There are finite Coxeter groups (certain dihedral groups and $H_3,H_4$) which do not admit a root system in the sense discussed here (called \emph{non-crystallographic groups}). We will discuss the extent to which our results apply to these groups in \Cref{sec:TypeH}.

The simplest infinite root systems are the affine root systems. An irreducible \newword{affine root system} is an irreducible root system such that there is a unique one-dimensional subspace of the span of $\Phi$ which pairs to $0$ with any element of $\Phi$ under the bilinear form. An affine root system is one whose irreducible components are all finite or affine, with at least one component affine. 
%(Some authors allow for there to be finite components.) % And maybe we do as well
In an irreducible affine root system, there is a unique vector $\delta\in\Spanp\Pi$ such that the intersection of $\Span_{\ZZ}\Pi$ with the one-dimensional subspace in the annihilator of $(-,-)$ is exactly $\ZZ\delta$. The nonzero integer multiples of $\delta$ are called \newword{imaginary roots} and $\delta$ is called the \newword{primitive imaginary root}. In this paper, we do not consider imaginary roots to be elements of the root system. (In other words, we work only with real root systems.)

There is a canonical way of turning a finite root system into an affine root system. Let $\Phi$ be a finite root system with span $V$. Then we construct a vector space $\widetilde V$ which is the direct sum of $V$ and a one-dimensional vector space spanned by a formal symbol $\delta$. We extend the inner product $(-,-)$ on $V$ to a bilinear form on $\widetilde V$ by declaring $(\delta,v) = 0$ for all $v\in \widetilde V$. Then we define $\widetilde\Phi$ to be the following subset of $\widetilde V$:
\[ \widetilde\Phi \coloneqq \{ \alpha + k\delta \mid \alpha\in \Phi,~ k\in\ZZ \}. \]
This is a root system, called the \newword{affinization} of $\Phi$, but its base is somewhat subtle to write down. If $\Phi$ is a finite irreducible root system then there is a unique root $\theta$, called the \newword{highest root}, such that $\alpha_i+\theta\not\in \Phi$ for all $\alpha_i\in \Pi$. If $\Phi$ is a reducible finite root system then there will be multiple highest roots $\theta_1,\ldots,\theta_m$ corresponding to the irreducible components of $\Phi$. Then a base for $\widetilde\Phi$ is given by
\[ \widetilde\Pi \coloneqq \Pi \sqcup \{ \delta-\theta_1, \ldots, \delta-\theta_m \}. \]
In particular, the rank of $\widetilde\Phi$ is always the rank of $\Phi$ plus the number of irreducible components of $\Phi$. Note that this implies that $\widetilde\Phi\hookrightarrow \widetilde V$ is geometrically embedded if and only if $\Phi$ is irreducible. It is also the case that $\widetilde\Phi$ is irreducible if and only if $\Phi$ is irreducible. In this case we write $\alpha_0$ for the fundamental root $\delta-\theta$. 

\begin{figure}
	\centering
	\begin{tabular}{|c|c|c|c|}
		\hline
		Name & Diagram & Name & Diagram \\
		\hline 
		$\tA_1$\rule{0pt}{2.6ex} & \dynkin[scale=1.5,extended,label]{A}{1} & $\tE_6$ & \dynkin[scale=1.5,label,extended]%,label,extended]
		{E}{6} \\
		\hline
		$\tA_n$, $n \geq 2$ & \dynkin[scale=1.5,extended,labels={0,1,2,n-1,n},edge length=.5cm]{A}{}	 & $\tE_7$ & \dynkin[scale=1.5,label,extended]{E}{7} \\
		\hline
		$\tB_n$ & \dynkin[scale=1.5,extended,labels={0,1,2,3,n-2,n-1,n},edge length=.5cm]{B}{} & $\tE_8$ & \dynkin[scale=1.5,label,extended]{E}{8} \\
		\hline
		$\tC_n$\rule{0pt}{2.6ex} & \dynkin[scale=1.5,extended,labels={0,1,2,n-2,n-1,n},edge length=.5cm]{C}{}  & $\tF_4$\rule{0pt}{2.6ex} & \Tstrut \Bstrut \dynkin[scale=1.5,label,extended]{F}{4} \\
		\hline
		$\tD_n$ & \dynkin[scale=1.5,extended,labels={0,1,2,3,n-3,n-2,n-1,n},label directions={,,,,,right,,}, edge length = .4cm]{D}{} & $\tG_2$\rule{0pt}{2.6ex} & \dynkin[scale=1.5,label,extended]{G}{2} \\
		\hline
	\end{tabular}
	\caption{The extended Dynkin diagrams associated to untwisted affine root systems.}\label{fig:untwisteddynkin}
\end{figure}

The irreducible affine root systems that arise from affinization are called the \newword{untwisted affine root systems}. Just like with finite root systems, there are diagrams describing the irreducible affine root systems called extended Dynkin diagrams. The untwisted extended Dynkin diagrams are shown in \Cref{fig:untwisteddynkin}, where the node associated to the ``extra root'' $\alpha_0$ is colored white. These diagrams describe Cartan matrices using the same edge rules as Dynkin diagrams; there is one new case coming from the root system $\tA_1$, which has Cartan matrix entries $a_{01}=a_{10}=-2$. Note that the subscript on the name of an affine diagram is one less than the rank of the affine root system, so that, for instance, $\tA_3$ has rank 4. 

Every other irreducible affine root system can be obtained from the untwisted affine root systems by taking a $\widetilde\Phi$, partitioning it into subsets (at most three are needed) and rescaling all roots in each subset by a fixed constant. Such irreducible affine root systems which are not untwisted are called \newword{twisted}. %Their extended Dynkin diagrams are shown in \Cref{fig:twisteddynkin}. 
We will describe the specific examples relevant to us in \Cref{sec:TwistedAffine}.

If $X$ is an irreducible root system from \Cref{fig:finitedynkin} or \Cref{fig:untwisteddynkin}, and $\Phi$ is a root system which is abstractly isomorphic to $X$, then we say $\Phi$ is \newword{of type} $X$. If $\Phi$ is reducible and has irreducible components which are abstractly isomorphic to $X^1,\ldots,X^m$, then we say $\Phi$ is of type $X^1\times\cdots\times X^m$.

\subsection{Root  subsystems}\label{sec:subsystems}
Let $\Phi \subset V$ be a root system with Weyl group $W$.
We will call a subset $\Lambda$ of $\Phi$ a \newword{root subsystem} if, for any roots $\alpha$ and $\beta$ in $\Lambda$, the reflection $t_\alpha\beta$ is also in $\Lambda$.
We write $\Lambda_+$ for $\Lambda \cap \pos$. 
Any subset $Y$ of $\Phi$ is contained in a unique smallest root subsystem, which we call the  \newword{root subsystem generated by $Y$.} A root subsystem $F$ of $\Phi$ is called \newword{full} if for any $\alpha$ and $\beta$ in $F$, we have $F \cap \Span\{\alpha,\beta\} = \Phi\cap \Span\{\alpha,\beta\}$. Similarly, any subset $Y$ of $\Phi$ is contained in a unique smallest full subsystem, which we call the  \newword{full subsystem generated by $Y$.} A root subsystem is itself a root system. A subsystem of an affine root system is finite or affine, and an irreducible subsystem of an untwisted affine root system is finite or untwisted affine.

A root $\gamma \in \Lambda_+$ is called \newword{fundamental in $\Lambda$} if $\gamma$ is not a positive linear combination of other roots in $\Lambda$. Write $\Pi_{\Lambda}$ for the set of fundamental roots in $\Lambda$; then $\Pi_{\Lambda}$ is a base for $\Lambda$.
We write $W_{\Lambda}$ for the subgroup of $W$ generated by the reflections over the roots in $\Lambda$.
A result of Dyer~\cite{Dyer1990} states that $W_{\Lambda}$ is a Coxeter group, with simple generators the reflections over the fundamental roots of $\Lambda$.
%In particular, $\Lambda$ is generated as a root subsystem by its fundamental generators. 
In particular, the rank of $\Lambda$ as a root system is equal to the rank of $W_{\Lambda}$ as a Coxeter group. %However, the rank of $\Lambda$ may not equal the linear rank of $\Lambda$, even when $\Phi$ is geometrically embedded.%, in other words, the number of fundamental roots of $\Lambda$.

We note that the rank of $\Lambda$ may be more than the linear rank  $\dim \Span(\Lambda)$, even when $\Phi$ is geometrically embedded. For example, in $\tA_3$, take 
\[\Lambda = \{ \alpha_0+\alpha_1+k \delta,  \alpha_1+\alpha_2+k \delta,  \alpha_2+\alpha_3+k \delta,  \alpha_0+\alpha_3+k \delta : k \in \ZZ \}.\] 
The fundamental roots are $\{  \alpha_0+\alpha_1,  \alpha_1+\alpha_2,  \alpha_2+\alpha_3,  \alpha_0+\alpha_3 \}$, so $\Lambda$ has rank $4$, but \[(\alpha_0+\alpha_1) + (\alpha_2+\alpha_3) = (\alpha_1+\alpha_2) + (\alpha_0+\alpha_3),\] 
so $\dim \Span(\Lambda)$ is only $3$. This example is explained by the fact that 
\[  \{\pm (\alpha_1+\alpha_2), \pm(\alpha_2+\alpha_3)   \}\]
is a root subsystem of $A_3$ of type $A_1\times A_1$. The subsystem $\Lambda$ defined above is exactly the affinization of this subsystem, which results in $\Lambda$ having type $\tA_1\times \tA_1$ and hence having rank $4$.
This phenomenon of root subsystems failing to be geometrically embedded is commonplace among affine root systems, and is a major source of difficulty in the theory of extended weak order.

\subsection{Root posets}
Given a root system $\Phi$ in a vector space $V$, we define a partial order on the elements of $\Phi$ by asserting that $\alpha \leq \beta$ if and only if $\beta-\alpha \in \Spanp\Pi$. This is called the \newword{root order} or the \newword{root poset} on $\Phi$. In general this partial order could depend on the realization $\Phi\hookrightarrow V$, but if $\Phi\hookrightarrow V_1,V_2$ are both geometric embeddings then they will have the same root poset. When needed, we refer to this canonical partial order as the \emph{abstract} root poset. The root poset on $\Phi\hookrightarrow V$ is always a refinement of the abstract root poset. In any root poset, the fundamental roots are exactly the set of minimal elements in the order.

If $\Lambda$ is a root subsystem of $\Phi$, then the root poset of $\Phi\hookrightarrow V$ restricts to a refinement of the root poset of $\Lambda\hookrightarrow V$. We note for future use that any linear order refining the root poset on a rank 2 root system must have the fundamental roots as its lowest two elements.

To see that the induced ordering on $\Lambda$ may fail to coincide with the abstract root poset of $\Lambda$, consider the example from the last subsection. The abstract root poset on $\Phi=\tA_3$ restricts to a root poset on the full subsystem $\Lambda\cong \tA_1\times\tA_1$ which has more relations than the abstract root poset; for instance, $\alpha_0+\alpha_1 < \alpha_1+\alpha_2+\delta$ but in the abstract root poset on $\Lambda$ these two elements are incomparable. It is also possible for the ordering on a subsystem to refine even its (non-abstract) root poset: $\{ \alpha_0+\alpha_1,\alpha_1+\alpha_2+\delta \}$ is the set of positive roots of a full rank 2 subsystem of $\Phi$, and any root poset on this rank 2 subsystem makes the two positive roots incomparable. Hence the induced ordering from $\Phi$ is a strict refinement of any root poset on this subsystem.

\subsection{Clean sets of vectors and suitable orderings}\label{sec:clean}
Return to the general setting of a subset $X$ of a real vector space $V$. Eventually we will take $X=\pos$ for some root system $\Phi$. Recall that $X$ is \emph{clean} if any $B\subseteq X$ which is biclosed in $X$ is also weakly separable in $X$. (See \Cref{fig:Implications} for the implications of cleanliness.) %(and hence also biconvex in $X$) 

\begin{remark}\label{rem:cleanlit}
	If $X$ is a clean set, then we say its dual hyperplane arrangement is a \newword{clean arrangement}. Clean arrangements have received plenty of study, but have not been named until this point. In \cite{McConville2014} it was shown that simplicial arrangements and hypersolvable arrangements are both clean. Finite Coxeter arrangements have been known to be clean since earlier \cite{Dyer1994,Papi1994, Blessenohl2005}.
	It was observed in \cite{Hohlweg2016} that infinite rank 3 Coxeter arrangements seem to be clean. This would imply \Cref{conj:Dyer} for rank 3 Coxeter groups by the work of Labb\'e \cite[Section 2.4]{Berlin2013}. 
	%Our \Cref{conj:suitable} implies that all rank 3 Coxeter arrangements are clean.
	Our \Cref{conj:cleanrank3} asserts that all rank 3 Coxeter arrangements are clean. 
	For affine Coxeter groups, this follows from the authors' classification of biclosed sets \cite{Barkley2022} or from work of Weijia Wang and Matthew Dyer \cite{Wang2018}. 
\end{remark}

\begin{remark}\label{rem:cleanrootposet}
	In \Cref{thm:FiniteClean}, we will show that root poset order ideals of a finite root system are clean. In \cite{Abe2016} it was shown that so-called ideal arrangements (hyperplane arrangements dual to root poset order ideals) are \emph{formal arrangements}, which roughly means that linear dependence can be checked on rank 2 subsystems. 
	It is also known that simplicial and hypersolvable arrangements are formal (see, e.g., \cite{Moller2023}).
	Thus the major examples of real hyperplane arrangements which are formal are also clean. But there exist arrangements over a finite field which are formal, while clean arrangements only make sense over ordered fields. In other words, being formal is a property of (a realization of) a matroid, while being clean is a property of (a realization of) an oriented matroid. We propose that cleanliness is an ``oriented version'' of formality. 
	As further evidence, since this article was released it has been shown that finite real $K(\pi,1)$ arrangements are clean \cite{Yoshinaga2024}; it is also known that finite $K(\pi,1)$ arrangements are formal \cite[Theorem 4.2]{Falk1987}. We leave the precise relationship between these notions open for future work.
\end{remark}

We will define a subset $Y$ of $X$ to be a \newword{linear subset} if $Y = X \cap L$ for some linear subspace $L$ of $V$; if this occurs then, of course, we can take $L = \Span_{\RR}(Y)$. We define  the \newword{dimension of $Y$} to be $\dim \Span_{\RR}(Y)$.
We note that linear subsets of dimension $\leq 2$ are automatically clean. 

%We say that $F$ is a \newword{full} subset of $X$ if for any $\alpha,\beta$ in $Y$, the linear subset of $X$ generated by $\{\alpha,\beta\}$ is contained in $Y$.

We will be especially interested in sets $X$ that have lots of clean subsets. To make this precise, let $X$ be a finite or countable subset of $V$. We make the (strong) assumption that any 2-dimensional linear subset $Y$ of $X$ has \newword{fundamental vectors}; that is, there are vectors $\alpha,\beta\in Y$ such that $Y \subseteq \Spanp\{\alpha,\beta\}$ and $\alpha\not\in \Spanp Y\setminus\{\alpha\}$ and $\beta\not\in\Spanp Y\setminus \{\beta\}$.

Let $\gamma_1,\gamma_2,\gamma_3,\ldots$ be an ordering of $X$ and set $X_i\coloneqq \{ \gamma_1,\gamma_2,\ldots,\gamma_i \}$. We will say this ordering is \newword{suitable} if
\begin{enumerate}
	\item In every $2$-dimensional linear subset $Y$ of $X$, the fundamental vectors of $Y$ are ordered before all the other vectors, and
	\item For every index $i$ and for any $\alpha,\beta,\gamma$ in $X_i$, there is a full subset $F \subseteq X$ containing $\{\alpha,\beta,\gamma\}$ such that $F \cap X_i$ is clean.
\end{enumerate}

The following lemma will let us construct suitable orderings in \Cref{sec:SuitableStart} by reducing to the case of rank 3 root systems.
\begin{lemma}\label{thm:rank3existence}
	Let $\Phi$ be a root system and let $\alpha,\beta,\gamma\in \pos$. Then there is a full subsystem $F \subseteq \Phi$ which contains $\{\alpha,\beta,\gamma\}$ and has rank at most 3.
\end{lemma}
\begin{proof}
	Let $F$ be the minimal full subsystem of $\Phi$ containing $\{\alpha,\beta,\gamma\}$. We will prove $F$ has rank at most 3. Indeed, assume not. Then $F$ has rank $r>3$. Hence a geometric realization $F \hookrightarrow V$ will have dimension $r$. If we take $F' \coloneqq F \cap \Span\{\alpha,\beta,\gamma\}$, then $F'$ is a linear subset of dimension at most 3. Hence $F'$ is a proper subset of $F$, which has dimension $r$. But $F'$ is a full subsystem of $\Phi$ and is strictly contained in $F$, contradicting the minimality of $F$.
\end{proof}

\begin{remark}
	If we were to allow $\Phi$ to include imaginary roots, then this lemma would fail to be true. For instance, in our running example of $\Phi=\tA_3$, if we considered the set $X=\pos\sqcup \NN\delta$ and took $\alpha = \alpha_0+\alpha_1$, $\beta = \alpha_1+\alpha_2$, and $\gamma=\alpha_2+\alpha_3$, then the minimal full subset of $X$ containing $\{\alpha,\beta,\gamma\}$ would be $\Lambda_+\sqcup \NN\delta$, where $\Lambda$ is the subsystem of type $\tA_1\times\tA_1$ from earlier examples. In particular, $\Lambda$ is a rank 4 root subsystem of $\Phi$.
\end{remark}

\section{Theorems on extending biclosed sets in the presence of a suitable ordering}
\label{sec:extend}
In this section, we will prove several major theorems about sets of vectors with suitable orderings. 
At this point in the paper, we have not presented any examples of sets of vectors with such orderings. We will turn to this issue in Sections~\ref{sec:FiniteSuitable},~\ref{sec:A2TildeSuitable} and~\ref{sec:Folding}. 
Sections~\ref{sec:FiniteSuitable} and~\ref{sec:A2TildeSuitable} can be read before this section, by the reader who would prefer to have examples first; 
the proofs in Section~\ref{sec:Folding} rely on the results in this section.

\begin{Theorem} \label{thm:KeyExtensionLemma}
	Let $X$ be a finite or countable subset of $V$ with a suitable ordering.
	Let $X_m$ be an initial section of the suitable ordering and let $U$ be co-closed in $X_m$. 
	Let $\overline{U}$ be the closure of $U$ in $X$. Then $\overline{U}$ is biclosed in $X$.
\end{Theorem}

\begin{proof}
	We inductively define a sequence of sets $V_i \subseteq X_i$ as follows: We have $V_0 = X_0 = \emptyset$.  We put $V_i = V_{i-1} \cup \{ \gamma_i \}$ if either 
	\begin{enumerate}
		\item $\gamma_i$ is in $U$ \textbf{or}
		\item There are $\zeta_1$ and $\zeta_2 \in V_{i-1}$ with $\gamma_i \in \Span_+(\zeta_1, \zeta_2)$.
	\end{enumerate}
	Otherwise, we put $V_i = V_{i-1}$.
	
	It is clear that $V_i \supseteq X_i \cap U$ for all $i$.
	
	Our first task is to show, by induction on $i$, that $V_i$ is biclosed in $X_i$. The base case, $V_0 = X_0 = \emptyset$, is obvious.
	We now move to the inductive case:
	
	\textbf{Case 1:} $\gamma_i \in U$. In this case, since $U \subseteq X_m$, we must have $i \leq m$.
	
	\textbf{Verification that $V_i$ is closed in $X_i$:} The thing that could go wrong is that there could be some $\alpha$, $\beta \in X_{i-1}$ with $\beta \in \Span_+(\alpha, \gamma_i)$, $\alpha \in V_{i-1}$ and $\beta \not\in V_{i-1}$.  Suppose that this is the case.
	
	Let $\phi$ and $\psi$ be the fundamental vectors in the linear subset $\Span(\alpha, \gamma_i) \cap X$, with $\alpha$ closer to the $\phi$ end of the subspace, and $\gamma_i$ closer to the $\psi$ end. Since our ordering is suitable, $\phi$ and $\psi$ must be ordered before $\beta$ (which is not a fundamental vector), so $\phi$ and $\psi$ are in $X_{i-1}$. Our inductive hypothesis states that $V_{i-1}$ is biclosed in $X_{i-1}$, and we have assumed that $\alpha \in V_{i-1}$ and $\beta \not\in V_{i-1}$, so we must have $\phi \in V_{i-1}$ and $\psi \not\in V_{i-1}$. Since $V_{i-1} \supseteq X_{i-1} \cap U$, we deduce that $\beta$ and $\psi \not\in U$. But then $\{ \beta, \gamma_i, \psi \}$ violates the hypothesis that $U$ is coclosed in $X_m$.
	
	\textbf{Verification that $V_i$ is co-closed in $X_i$:} The thing that could go wrong is that there could be some $\alpha$, $\beta \in X_{i-1}$ with $\gamma_i \in \Span_+(\alpha, \beta)$ and $\alpha$, $\beta \not\in V_i$. But then $\alpha$, $\beta \not \in U$, so $\{ \alpha, \gamma_i, \beta \}$ violates the hypothesis that $U$ is co-closed in $X_m$. 
	
	\textbf{Case 2:} $\gamma_i$ is not in $U$ but there are $\zeta_1$ and $\zeta_2 \in V_{i-1}$ with $\gamma_i \in \Span_+(\zeta_1, \zeta_2)$.
	
	\textbf{Verification that $V_i$ is closed in $X_i$:}  The thing that could go wrong is that there could be some $\alpha$, $\beta \in X_{i-1}$ with $\beta \in \Span_+(\alpha, \gamma_i)$, $\alpha \in V_{i-1}$ and $\beta \not\in V_{i-1}$.  Suppose that this is the case.
	
	Let $F$ be the full subset containing $\{ \alpha, \zeta_1, \zeta_2\}$ such that $F \cap X_{i-1}$ is clean. 
	Since $F$ is full, $\gamma_i$ and $\beta$ are also in $F$.  
	But we know by induction that $V_{i-1}$ is biclosed in $X_{i-1}$, so $V_{i-1} \cap F$ is biclosed in $X_{i-1} \cap F$. 
	But the hypothesis of suitability then says that $V_{i-1} \cap F$ should be separable in $X_{i-1} \cap F$, and this violates that $\alpha$, $\zeta_1$, $\zeta_2 \in V_{i-1} \cap F$ and $\beta \in (X_{i-1} \cap F) \setminus V_{i-1}$.
	
	\textbf{Verification that $V_i$ is co-closed in $X_i$:} The thing that could go wrong is that there could be some $\alpha$, $\beta \in X_{i-1}$ with $\gamma_i \in \Span_+(\alpha, \beta)$ and $\alpha$, $\beta \not\in V_i$. Suppose that this is the case.
	
	Let $F$ be the full subset containing $\{ \zeta_1, \gamma_i, \alpha \}$ such that $F \cap X_{i-1}$ is clean. Since $F$ is full, $\zeta_2$ and $\beta$ are also in $F$. 
	But we know by induction that $V_{i-1}$ is biclosed in $X_{i-1}$, so $V_{i-1} \cap F$ is biclosed in $X_{i-1} \cap F$. 
	But the hypothesis of suitability then says that $V_{i-1} \cap F$ should be separable in $X_{i-1} \cap F$, and this violates that  $\zeta_1$, $\zeta_2 \in V_{i-1} \cap F$ and $\alpha$, $\beta \in (X_{i-1} \cap F) \setminus V_{i-1}$.
	
	\textbf{Case 3:} $V_i = V_{i-1}$. 
	
	\textbf{Verification that $V_i$ is closed in $X_i$:}  The thing that could go wrong is that there could be $\zeta_1$, $\zeta_2 \in V_{i-1}$ and $\gamma_i \in \Span_+(\zeta_1, \zeta_2)$. But then we would have $\gamma_i \in V_i$ after all.
	
	\textbf{Verification that $V_i$ is co-closed in $X_i$:} The thing that could go wrong is that there could be $\alpha$, $\beta \in X_{i-1}$ with $\beta \in \Span_+(\alpha, \gamma_i)$, with $\beta \in V_{i-1}$ and $\alpha \not\in V_{i-1}$.
	
	Let $\phi$ and $\psi$ be the fundamental vectors in the linear subset $\Span(\alpha, \gamma_i) \cap X$, with $\alpha$ closer to the $\phi$ end of the subspace, and $\gamma_i$ closer to the $\psi$ end. Since our ordering is suitable, $\phi$ and $\psi$ must be ordered before $\beta$ (which is not a fundamental vector), so $\phi$ and $\psi$ are in $X_{i-1}$.
	Our inductive hypothesis states that $V_{i-1}$ is biclosed in $X_{i-1}$, and we have assumed that $\alpha \not\in V_{i-1}$ and $\beta \in V_{i-1}$. So we must have $\phi \not\in V_{i-1}$ and $\psi \in V_{i-1}$. But then $\gamma_i \in \Span_+(\beta, \psi)$, and we have seen that $\beta$, $\psi \in V_{i-1}$, in which case we would have $\gamma_i \in V_i$ after all.
	
	This concludes the inductive verification that $V_{i}$ is biclosed in $X_{i}$. Therefore, $\bigcup V_i$ is biclosed in $\bigcup X_i = X$. We will therefore be done if we can show that $\bigcup V_i$ is the closure of $U$ in $X$. 
	
	Every time that we put the vector $\gamma_i$ into $V_i$, it is either because $\gamma_i \in U$ or because $\gamma_i$ is in the positive span of two vectors in $V_{i-1}$. So, inductively, we have $V_i \subset \overline{U}$ for all $i$, and thus $\bigcup V_i \subseteq \overline{U}$. 
	But we also have proved that $\bigcup V_i$ is closed in $X$, and clearly $U \subseteq \bigcup V_i$, so we must have $\overline{U} \subseteq \bigcup V_i$. 
	This shows that $\bigcup V_i = \overline{U}$, and concludes the proof.
\end{proof}

We now pursue variants of Theorem~\ref{thm:KeyExtensionLemma}.

\begin{prop} \label{prop:BiclosedExtend}
	With notation as in Theorem~\ref{thm:KeyExtensionLemma}, assume that $U$ is not only co-closed in $X_m$, but that $U$ is biclosed in $X_m$. Then $\overline{U} \cap X_m = U$.
\end{prop}

\begin{proof}
	Clearly, $U \subseteq \overline{U}$.
	We need to show, for $1 \leq i \leq m$, that, if the vector $\gamma_i$ is in $\overline{U}$, then $\gamma_i \in U$. 
	Suppose that this is not true and let $i$ be the least index for which $\gamma_i \in \overline{U} \setminus U$. Then we must have $\gamma_i \in \Span_+(\zeta_1, \zeta_2)$ for some $\zeta_1$, $\zeta_2$ in $\overline{U} \cap X_{i-1}$. So $\zeta_1$ and $\zeta_2$ are $\gamma_a$ and $\gamma_b$ for some $a$, $b < i$. Using the minimality of $i$, we have $\gamma_a$ and $\gamma_b \in U$. But, since $U$ is closed in $X_m$, this implies that $\gamma_i$ is in $U$ as well.
\end{proof}

\begin{prop} \label{prop:ClosureOfCoclosed}
	Let $X$ be a finite or countable subset of $V$ with a suitable ordering.
	Let $U$ be co-closed in $X$, and let $\overline{U}$ be the closure of $U$ in $X$. Then $\overline{U}$ is biclosed in $X$.
\end{prop}

\begin{proof}
	The proof follows exactly as in Theorem~\ref{thm:KeyExtensionLemma}, simply taking $m = \infty$.
\end{proof}

We record the corresponding dual statement.
\begin{prop} \label{prop:InteriorOfClosed}
	Let $X$ be a finite or countable subset of $V$ with a suitable ordering.
	Let $K$ be closed in $X$ and let $K^{\circ}$ be the interior of $K$ in $X$. Then $K^{\circ}$ is biclosed in $X$.
\end{prop}

\begin{proof}
	Take $K = X \setminus U$ and apply Proposition~\ref{prop:ClosureOfCoclosed}.
\end{proof}

We note one more result that will be used to prove Dyer's Conjecture A.
\begin{lemma} \label{thm:ConjectureAFinite}
	Let $X$ be a finite or countable subset of $V$ with a suitable ordering.
	Let $X_m$ be a finite initial section of the suitable ordering and let $B_1\subset B_2$ be biclosed sets in $X_m$. If $|B_2\setminus B_1| \geq 2$, then there is a biclosed set $C$ in $X_m$ with $B_1\subsetneq C \subsetneq B_2$.
\end{lemma}
\begin{proof}
	We will show by induction on $m$ that if $B_1\lessdot B_2$ is a cover relation in the poset of biclosed sets in $X_m$, then $|B_2\setminus B_1|= 1$. Assume this is true for $X_i$ with $i<m$, and assume for the sake of contradiction that $B_1\lessdot B_2$ is a cover relation of biclosed sets in $X_m$ with $|B_2\setminus B_1| \geq 2$. By induction we know the restrictions $B_1\cap X_{m-1}$ and $B_2\cap X_{m-1}$ satisfy $|(B_2\cap X_{m-1}) \setminus (B_1\cap X_{m-1})| \leq 1$, so it follows that there is a unique root $\alpha \in X_{m-1}$ such that $B_2 = B_1\sqcup \{\alpha,\gamma_m\}$. We will derive a contradiction by showing that there is a biclosed set $C$ strictly between $B_1$ and $B_2$. To decide whether $C = B_1\cup \{\alpha\}$ or $C=B_1\cup\{\gamma_m\}$, examine the 2-dimensional linear subset $Y=\Span(\alpha,\gamma_m)\cap X_m$. If $|Y| = 2$, then we set $C\coloneqq B_1\cup\{\alpha\}$. Otherwise, $\gamma_m$ is not a fundamental vector of $Y$, since at least $2$ vectors in $Y$ precede $\gamma_m$ in a suitable order. In this case, there is a unique choice among the two options for $C$ such that $C\cap Y$ is biclosed in $Y$, and this is the one we pick. 
	
	Let $\alpha_C$ be either $\alpha$ or $\gamma_m$, so that $C=B_1\cup\{\alpha_C\}$. We claim that $C$ is biclosed in $X_m$. If not, then either there is an element $\beta\in C$ and a $\gamma \in \Spanp(\alpha_C, \beta)\cap X_m$ such that $\gamma\not\in C$, or else there are elements $\beta,\gamma \in X_m\setminus C$ so that $\alpha_C\in \Spanp(\beta,\gamma)$. In either case, let $F$ be a full, clean, subset of $X_m$ containing $\beta, \alpha, \gamma_m$. Write $B_1'=B_1\cap F$ and $B_2'=B_2\cap F$. Then $B_1'$ and $B_2'$ are separable, and $B_2'= B_1'\sqcup \{\alpha,\gamma_m\}$. Fullness of $F$ implies that $Y\subseteq F$. If $|Y|>2$, then there is a unique separable set $C'$ in $F$ which is strictly between $B_1'$ and $B_2'$. Since $C'$ must restrict to a biclosed set in $Y$, it follows that $C' = C\cap F$, contradicting the hypothesis that $C\cap F$ is not biclosed. 
	
	Otherwise, $|Y|=2$. We claim that in this case, both $B_1'\cup\{\alpha\}$ and $B_1'\cup\{\gamma_m\}$ are separable in $F$, contradicting the hypothesis that $C\cap F$ is not biclosed. To see this, note that at least one of $B_1'\cup\{\alpha\}$ and $B_1'\cup\{\gamma_m\}$ is separable; without loss of generality, $B_0\coloneqq B_1'\cup\{\alpha\}$ is separable. We will show $B_1'\cup \{\gamma_m\}$ is biclosed (and hence separable); note that it is enough to check this on two dimensional subsets of $F$. Consider any two dimensional subset $Z$ of $F$. If $Z\cap \{\alpha,\gamma_m\}=\varnothing$, then $Z\cap (B_1'\cup \{\gamma_m\})=Z\cap B_0$ so is biclosed in $Z$. If $Z\cap \{\alpha,\gamma_m\} = \{\alpha\}$, then $Z\cap (B_1'\cup \{\gamma_m\}) = Z\cap B_1'$ so is biclosed in $Z$. If $Z\cap \{\alpha,\gamma_m\} = \{\gamma_m\}$, then $Z\cap (B_1'\cup \{\gamma_m\}) = Z\cap B_0$ so is biclosed in $Z$. Finally, if $\{\alpha,\gamma_m\}\subseteq Z$, then $Z = \{\alpha,\gamma_m\}$ and any subset of $Z$ is biclosed in $Z$.
\end{proof}

\section{Suitable orderings imply Dyer's conjectures}
\label{sec:biclosedlattice}

Propositions~\ref{prop:ClosureOfCoclosed} and~\ref{prop:InteriorOfClosed} imply that, if $X$ has a suitable ordering, then the biclosed sets of $X$ form a lattice.
In this section, we state this result more precisely, and recall the proof. We also indicate how the results of the previous section imply Dyer's Conjecture A.

\begin{theorem} \label{thm:BiclosedLattice}
	Let $X$ be a finite or countable subset of $V$ with a suitable ordering.
	Then the biclosed subsets of $X$ form a complete lattice with respect to containment.
	More specifically, if $\cX$ is any collection of biclosed subsets of $X$, then we have the following formulas for meet and join:
	\[ \bigjoin_{X \in \cX} X = \overline{\bigcup_{X \in \cX} X} \qquad \bigmeet_{X \in \cX} X = \left(\bigcap_{X \in \cX} X \right)^{\circ} . \]
\end{theorem}

\begin{proof}
	We prove that $\overline{\bigcup_{X \in \cX} X}$ is the join of $\cX$; the statement about meets is similar.
	We first need to know that $\overline{\bigcup_{X \in \cX} X}$  is biclosed in the first place!
	
	Set $U := \bigcup_{X \in \cX} X$. Since each $X$ in the union is co-closed, the union is also co-closed. By Proposition~\ref{prop:ClosureOfCoclosed}, we conclude that $\overline{U}$ is biclosed.
	
	Now, let $B$ be any biclosed set containing all of the $X$ in $\cX$. Then $B$ contains the union $U$. Since $B$ is closed, we have $B \supseteq \overline{U}$. Thus, we have shown that $\overline{U}$ is the least upper bound for $\cX$.
\end{proof}

\begin{theorem} \label{thm:BiclosedExtend}
	Let $X$ be a finite or countable subset of $V$ with a suitable ordering and let $X_m$ be an initial section of that suitable ordering.
	Let $U$ be biclosed in $X_m$. Then $\overline{U}$ is biclosed in $X$, and $\overline{U} \cap X_m = U$.
\end{theorem}

\begin{proof}
	This follows immediately from Proposition~\ref{prop:BiclosedExtend}.
\end{proof}

\begin{theorem}[Dyer's ``Conjecture A'']\label{thm:ConjectureA}
	Let $X$ be a finite or countable subset of $V$ with a suitable ordering. Let $\cC=\{B_i\}_{i\in I}$ be a maximal chain in the poset of biclosed sets: $\cC$ is totally ordered by containment of biclosed sets and is a maximal family with this property. Then for any $\alpha, \beta \in X$, there is a $B_{i}$ in $\cC$ such that $|\{\alpha,\beta\}\cap B_{i}|=1$.
\end{theorem}
\begin{proof}
	Consider the intersection $B_2$ of all elements of $\cC$ which contain $\{\alpha,\beta\}$. This is a biclosed set since it is the intersection of a decreasing sequence of biclosed sets. Furthermore $B_2$ must be an element of $\cC$ since it is the greatest lower bound of a subset of $\cC$. We similarly let $B_1$ be the largest element of $\cC$ which is disjoint from $\{\alpha,\beta\}$. The claim is equivalent to the existence of a biclosed set in $X$ which is strictly between $B_1$ and $B_2$. Let $X_m$ be a finite initial section of $X$ which contains $\alpha$ and $\beta$. By \Cref{thm:ConjectureAFinite}, there is a biclosed set $C$ in $X_m$ which is strictly between $B_1\cap X_m$ and $B_2\cap X_m$. Let $\overline{C}$ be the closure of $C$ in $X$; by \Cref{thm:KeyExtensionLemma}, $\overline{C}$ is biclosed in $X$. 
	
	We claim that $B_1\vee \overline{C}$ is strictly between $B_1$ and $B_2$. Clearly, this join is strictly greater than $B_1$. It is also at most $B_2$, since $B_1,\overline{C}\leq B_2$. So we need to show $B_1\vee \overline{C}\neq B_2$, which follows since 
	\[(B_1\vee\overline{C})\cap X_m = (B_1\cap X_m)\vee C = C \neq B_2\cap I,\]
	where the latter join is computed in the biclosed sets of $X_m$. Hence there is a biclosed set strictly between $B_1$ and $B_2$, and the theorem follows.
\end{proof}

\section{Sets of vectors where every subset is clean} \label{sec:EasySets}

Let $X$ be a set of vectors. To construct suitable orderings of $X$, we need every triple of vectors $\{ \alpha, \beta, \gamma \}$ of $X$ to be contained in a full subset $Y$ where every initial section of the order intersects $Y$ in a clean set.
The easiest way to do this is if every  triple of vectors is contained in a full subset $Y$ where every subset of $Y$ is clean.
In this section, we will discuss some cases where that occurs.

\begin{lemma} \label{lem:2DimEasy}
	Let $Y$ be a set of vectors contained in a two dimensional linear subset. Then every subset of $Y$ is clean.
\end{lemma}

\begin{proof}
	This is obvious.
\end{proof}

We will say that a set of vectors $Y$ is \newword{disconnected} if we can write $Y$ as $Y_1 \sqcup Y_2$ with $Y_1$ and $Y_2$ nonempty and full in $Y$; we will call $(Y_1, Y_2)$ a \newword{decomposition of $Y$}. 
We will call $Y$ \newword{connected} if $Y$ is not disconnected.

\begin{lemma} \label{lem:DisconnectedEasy}
	Suppose that $Y$ is disconnected, with decomposition $(Y_1, Y_2)$, and that both $Y_1$ and $Y_2$ have the property that every triple of vectors in $Y_i$ is contained in a full subset, every subset of which is clean. Then $Y$ also has the property that every triple of vectors in $Y$, is contained in a full subset, every subset of which is clean. 
\end{lemma}

\begin{proof}
	Let $\{ \alpha, \beta, \gamma \}$ be a triple of vectors in $Y$. If $\{ \alpha, \beta, \gamma \}$ are all in $Y_1$, or all in $Y_2$, then we are done.
	Otherwise, without loss of generality, let $\alpha$, $\beta \in Y_1$ and $\gamma \in Y_2$. 
	
	Let $L$ be the $2$-dimensional linear subset of $Y$ containing $\{ \alpha, \beta \}$. We claim that $L \cup \{ \gamma \}$ is full in $Y$.
	The way that this could fail is if there is some $\phi \in L$ and some $\psi \in Y$ in $\Span(\phi, \gamma)$ other than $\phi$, $\gamma$.
	Since $\phi \in L$, we have $\phi \in Y_1$. If $\psi \in Y_1$, then note that $\gamma \in \Span(\phi, \psi)$, so the fullness of $Y_1$ in $Y$ implies that $\gamma \in Y_1$, contradicting that $\gamma \in Y_2$. 
	Alternatively, if $\psi \in Y_2$, then note that $\phi \in \Span(\gamma, \psi)$, so the fullness of $Y_2$ in $Y$ implies that $\phi \in Y_1$, contradicting that $\phi \in Y_1$. 
	
	Thus, $L \cup \{ \gamma \}$ is a full subset of $Y$ containing $\{ \alpha, \beta, \gamma \}$.
	It is obvious that every subset of $L \cup \{ \gamma \}$ is clean.
\end{proof}

\section{Preliminaries for constructing suitable orders of root systems} \label{sec:SuitableStart}

We now have a substantial list of results which will apply if we can find suitable orderings of root systems.
In Sections~\ref{sec:FiniteSuitable} and \ref{sec:A2TildeSuitable}, we will prove:
\begin{theorem} \label{thm:SimplyLacedOrdering}
	Let $\Phi$ be a finite or affine simply-laced root system. 
	Take any ordering $\gamma_1$, $\gamma_2$, $\gamma_3$, \dots of $\pos$ which refines the root poset on $\pos$. Then this ordering is suitable.
\end{theorem}
In Section~\ref{sec:Folding}, we will prove the same result without the simply laced hypothesis:
\begin{theorem} \label{thm:NonSimplyLacedOrdering}
	Let $\Phi$ be a crystallographic finite root system, or a non-twisted affine root system. 
	Take any ordering $\gamma_1$, $\gamma_2$, $\gamma_3$, \dots of $\pos$ which refines the root poset on $\pos$. Then this ordering is suitable.
\end{theorem}
See Section~\ref{sec:TwistedAffine} for difficulties of the twisted affine case, and Section~\ref{sec:TypeH} for some thoughts on the finite non-crystallographic types.

\begin{remark} 
	The notations $\tB_2$ and $\tC_2$ both denote the same root system (just as $B_2$ and $C_2$ do). We have chosen to call it $\tC_2$ on aesthetic grounds: The two edges of the Coxeter diagram labeled $4$ make it look more like a type $\tC$ diagram than a type $\tB$ diagram to us.
\end{remark}

In this section, we discuss the commonalities of the proofs in all the cases. 
For any $\alpha$, $\beta$, $\gamma$ in $\pos$, we need to find a full subset $F$ of $\pos$ containing $\alpha$, $\beta$, $\gamma$ such that every initial section of $F$ is clean.
Let $\Lambda$ be the minimal full root subsystem containing $\{ \alpha, \beta, \gamma \}$. By \Cref{thm:rank3existence}, $\Lambda$ has rank at most 3. It is critical that we do not take $\Lambda$ to be the minimal linear subset containing $\{\alpha,\beta,\gamma\}$, since we have seen by example that this may be a root subsystem of rank 4.

If $\Lambda$ is rank 2, then Lemma~\ref{lem:2DimEasy} applies and we are done.
If $\Lambda$ is a reducible root system, then each factor will have dimension $\leq 2$, so Lemma~\ref{lem:DisconnectedEasy} applies and we are done.

This leaves the cases where $\Lambda$ an irreducible rank 3 root system. Since we know $\Lambda$ is finite or untwisted affine, it follows that $\Lambda$ is of type $A_3$, $B_3$, $C_3$, $\tA_2$, $\tC_2$, or $\tG_2$.
Any order ideal of the root poset on $\pos$ will restrict to an order ideal of the root poset on a subsystem.
Thus, our goal is to prove the following:

\begin{lemma} \label{lem:KeyCleanLemma}
	Let $\Lambda$ be a root system of type $A_3$, $B_3$, $C_3$, $\tA_2$, $\tC_2$, or $\tG_2$. 
	Let $J$ be a finite order ideal of $\Lambda_+$. Then $J$ is clean.
\end{lemma}

In every type, we will prove Lemma~\ref{lem:KeyCleanLemma} by induction on $\# J$. 
The base case, $J = \emptyset$, is obvious.
Thus, suppose that we are trying to prove the lemma for some finite order ideal $J$, let $\gamma$ be a maximal element of $J$, and put $J' = J \setminus \{ \gamma \}$. 
So, inductively, we know that $J'$ is clean.

Let $B$ be biclosed in $J$ and put $B' = B \cap J'$. 
So $B'$ is biclosed in $J'$ and, by induction, we know that $B'$ is separable.
Let 
\[ \Omega = \{ \theta \in V^{\ast} : \langle \beta, \theta \rangle < 0 \ \text{for} \ \beta \in B',\ \langle \beta, \theta \rangle > 0 \ \text{for} \ \beta \in J' \setminus B' \}. \]
The assumption that $B'$ is separable means that $\Omega$ is nonempty.
We note that $\Omega$ determines the set $B'$ by $B' = \{ \beta \in J' : \langle \beta,- \rangle <0 \ \text{on} \ \Omega \}$.
If the hyperplane $\gamma^{\perp}$ passes through $\Omega$, then both $B' \cup \{ \gamma \}$ and $B'$ are separable in $J$, so we are done. 
So we need to deal with the cases that $\langle \gamma,-\rangle$ is entirely positive or entirely negative on $\Omega$. 

Thus, in order to prove Lemma~\ref{lem:KeyCleanLemma} for a root system $\Lambda$, we need to prove the following:
\begin{lemma} \label{lem:KeyCleanLemmaRephrase}
	Let $\gamma$ be a root in $\Lambda_+$ and let $J' \sqcup \{ \gamma \}$ be an order ideal in $\Lambda_+$ where $\gamma$ is maximal. 
	Let $\Omega$ be a region of the hyperplane arrangement $\bigcup_{\beta \in J'} \beta^{\perp}$; and let $B' =  \{ \beta \in J' : \langle \beta,-\rangle <0 \ \text{on} \ \Omega \}$.
	If $\langle \gamma,-\rangle$ is negative on all of $\Omega$, then $\gamma$ is in the closure of $B'$; if  $\langle \gamma,-\rangle$ is positive on all of $\Omega$, then $\gamma$ is in the interior of $J' \setminus B'$.
\end{lemma}

%Replacing $\Omega$ by $-\Omega$, we may consider only the case that  $\langle \gamma,\ \rangle$ is negative on all of $\Omega$. 
Lemma~\ref{lem:KeyCleanLemmaRephrase} is what we will check in each root system. 
Note that, since we are assuming that $J' \sqcup \{ \gamma \}$ is an order ideal, where $\gamma$ is maximal, the set $J'$ must contain $\{ \beta \in \Lambda_+ : \beta \prec \gamma \}$ and must be contained in $\{ \beta \in \Lambda_+ : \beta \not\succeq \gamma \}$

\section{Verification of Lemma~\ref{lem:KeyCleanLemmaRephrase} in type $A_3$} \label{sec:FiniteSuitable}

The goal of this section is to verify Lemma~\ref{lem:KeyCleanLemmaRephrase} in type $A_3$.
We first explain the meaning of the figures in these proofs. 
Take the $A_3$ hyperplane arrangement and intersect it with a $2$-sphere around the origin, to obtain an arrangement of great circles on the $2$-sphere. We draw these circles in a stereographic projection, as shown in Figure~\ref{fig:A3Arrangement}.
We label the fundamental domain, $D$, with a $D$ in our figures. So moving towards $D$ is moving down in weak order.

\begin{figure}[h]
	\begin{tikzpicture}
		\draw (2.25,2.25) -- (-2.25,-2.25) node[left] {$\alpha_1^{\perp}$};
		\draw (-2.25,2.25) -- (2.25,-2.25) node[right] {$\alpha_3^{\perp}$};
		\draw (0,-0.5) circle (1.25) node [xshift=20, yshift=-35] {$\alpha_2^{\perp}$};
		\draw (-0.5,0) circle (1.25) node [xshift=-65] {$(\alpha_1+\alpha_2)^{\perp}$};
		\draw (0.5,0) circle (1.25) node [xshift=65] {$(\alpha_2+\alpha_3)^{\perp}$};
		\draw (0,0.5) circle (1.25) node [yshift=45] {$(\alpha_1+\alpha_2+\alpha_3)^{\perp}$};
		\draw (0, -2.1) node {$\boldsymbol{D}$};
	\end{tikzpicture}
	\caption{The $A_3$ hyperplane arrangement.} \label{fig:A3Arrangement}
\end{figure}
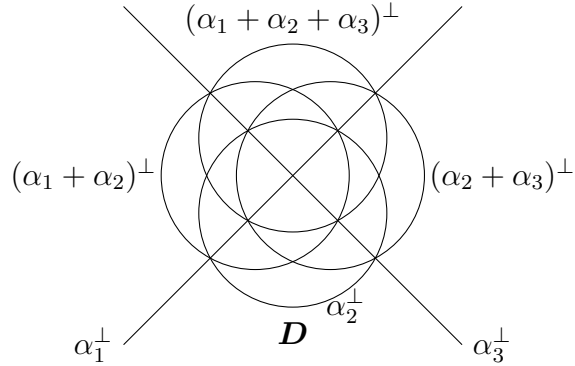

%The solid great circles are dual to the roots in the order ideal $J'$, the dashed circles are dual to roots which may not be in $J'$, the thick circle is dual to $\gamma$.

Recall that the positive roots in type $A_3$ are $\{ \alpha_1, \alpha_2, \alpha_3, \alpha_1+\alpha_2, \alpha_2+\alpha_3, \alpha_1+\alpha_2+\alpha_3 \}$.
Let $P= \{ \beta \in \Lambda_+ : \beta \prec \gamma \}$ and let $Q = \{  \beta \in \Lambda_+ : \beta \not\succeq \gamma \}$, so $P \subseteq J' \subseteq Q$.
We list the values of $P$ and $Q$ in the table below:
\[ \begin{array}{|c|c|c|}
	\hline
	\gamma & P & Q \\
	\hline
	\alpha_1 & \emptyset & \{ \alpha_2, \alpha_3, \alpha_2+\alpha_3 \} \\
	\alpha_2 & \emptyset & \{ \alpha_1, \alpha_3 \} \\
	\alpha_3 & \emptyset & \{ \alpha_1, \alpha_2, \alpha_1+\alpha_2 \} \\
	\alpha_1+\alpha_2 & \{ \alpha_1, \alpha_2 \} &  \{ \alpha_1, \alpha_2, \alpha_3, \alpha_2+\alpha_3 \} \\
	\alpha_2+\alpha_3 & \{ \alpha_2, \alpha_3 \} &  \{ \alpha_1, \alpha_2, \alpha_3, \alpha_1+\alpha_2 \} \\
	\alpha_1+\alpha_2+\alpha_3 & \{ \alpha_1, \alpha_2, \alpha_3, \alpha_1+\alpha_2, \alpha_2+\alpha_3 \} & \{ \alpha_1, \alpha_2, \alpha_3, \alpha_1+\alpha_2, \alpha_2+\alpha_3 \} \\
	\hline
\end{array}
\]
In the figures below, $\gamma^{\perp}$ is drawn in bold (and labeled), hyperplanes $\beta^{\perp}$ for $\beta \in P$ are drawn with normal thickness and hyperplanes $\beta^{\perp}$ for $\beta \in Q \setminus P$ are drawn dashed. 
%Our goal is to show that, if $B'$ is the separating set of a region on the negative side of $\gamma^\perp$ and which is not incident to $\gamma^\perp$, then $\gamma$ is in the closure of $B'$, and similarly for regions on the positive side of $\gamma^\perp$. 
To avoid clutter, we only include the labels $\beta^{\perp}$ for those $\beta$'s which are key to the current argument.

\textbf{Case 1: $\gamma$ is one of $\alpha_1$, $\alpha_2$, $\alpha_3$}. In this case, $\gamma^{\perp}$ crosses through every region in the hyperplane arrangement $\bigcup_{\beta \in Q} \beta^{\perp}$, so the Lemma is vacuously true.
The left-hand side of Figure~\ref{fig:Case1A3} depicts the case $\gamma = \alpha_1$ and the right-hand side depicts $\gamma = \alpha_2$.

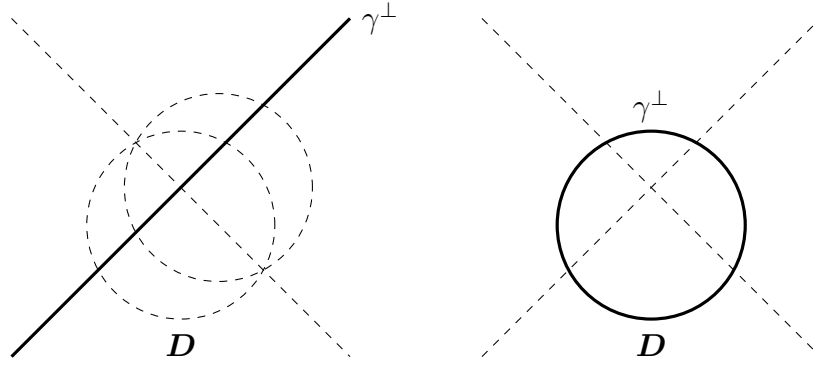
\begin{figure}
	\begin{tikzpicture}
		\draw[very thick] (-2.25,-2.25) -- (2.25,2.25) node[right] {$\gamma^{\perp}$}; %alpha1
		\draw[dashed] (0,-0.5) circle (1.25) node [yshift=45] {}; %alpha2
		\draw[dashed] (-2.25,2.25) -- (2.25,-2.25) node[right] {}; %alpha3
		%\draw (-0.5,0) circle (1.25) node [xshift=-65] {}; %alpha1+alpha2
		\draw[dashed] (0.5,0) circle (1.25) node [xshift=65] {};  %alpha2+alpha3
		%\draw (0,-0.5) circle (1.25) node [yshift=-40] {}; %alpha1+alpha2+alpha3
		\draw (0, -2.1) node {$\boldsymbol{D}$};
	\end{tikzpicture} \qquad
	\begin{tikzpicture}
		\draw[dashed] (-2.25,-2.25) -- (2.25,2.25) node[right] {}; %alpha1
		\draw[very thick] (0,-0.5) circle (1.25) node [yshift=45] {$\gamma^{\perp}$}; %alpha2
		\draw[dashed] (-2.25,2.25) -- (2.25,-2.25) node[right] {}; %alpha3
		%\draw[dashed] (-0.5,0) circle (1.25) node [xshift=-65] {}; %alpha1+alpha2
		%\draw[dashed] (0.5,0) circle (1.25) node [xshift=65] {};  %alpha2+alpha3
		%\draw (0,-0.5) circle (1.25) node [yshift=-40] {}; %alpha1+alpha2+alpha3
		\draw (0, -2.1) node {$\boldsymbol{D}$};
	\end{tikzpicture} 
	\caption{Case 1 in the proof of Lemma~\ref{lem:KeyCleanLemmaRephrase} for $A_3$.} \label{fig:Case1A3}
\end{figure}

\textbf{Case 2: $\gamma$ is one of $\alpha_1+\alpha_2$, $\alpha_2+\alpha_3$}. These two cases are symmetric to each other, we discuss the case $\gamma = \alpha_1 + \alpha_2$, which we depict in Figure~\ref{fig:Case2A3}.
By symmetry, we only have to consider regions of $\bigcup_{\beta \in J'} \beta^{\perp}$ which lie entirely on the negative side of $\gamma^{\perp}$. There are three of these  (shaded in gray) if $J' = Q$, which may merge into fewer regions if $J'$ is smaller. The corresponding $B'$ sets are $J' \cap \{ \alpha_1, \alpha_2 \}$, $J' \cap \{ \alpha_1, \alpha_2, \alpha_2+\alpha_3 \}$ and $J' \cap \{ \alpha_1, \alpha_2, \alpha_3, \alpha_2+\alpha_3 \}$. 
In every case, we have $\alpha_1$ and $\alpha_2 \in B'$, so  $\alpha_1+\alpha_2$ is in the closure of $B'$ as required.

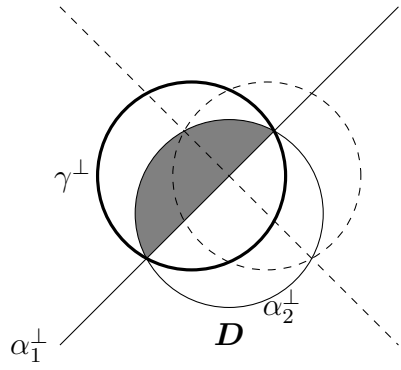
\begin{figure}[h]
	\begin{tikzpicture}
		\begin{scope}
			\path[clip] (2.25, 2.25) -- (-2.25, -2.25) -- (-2.25, 2.25) -- cycle;
			\fill[gray]  (0,-0.5) circle (1.25) ; 
		\end{scope}
		\draw (2.25,2.25) -- (-2.25,-2.25) node[left] {$\alpha_1^{\perp}$}; %alpha1
		\draw (0,-0.5) circle (1.25) node [xshift=20, yshift=-35] {$\alpha_2^{\perp}$}; %alpha2
		\draw[dashed] (-2.25,2.25) -- (2.25,-2.25) node[right] {}; %alpha3
		\draw[very thick] (-0.5,0) circle (1.25) node [xshift=-45] {$\gamma^{\perp}$}; %alpha1+alpha2
		\draw[dashed] (0.5,0) circle (1.25) node [xshift=-65] {};  %alpha2+alpha3
		%\draw (0,0.5) circle (1.25) node [yshift=-40] {}; %alpha1+alpha2+alpha3
		\draw (0, -2.1) node {$\boldsymbol{D}$};
	\end{tikzpicture}
	\caption{Case 2 in the proof of Lemma~\ref{lem:KeyCleanLemmaRephrase} for $A_3$.} \label{fig:Case2A3}
\end{figure}

\textbf{Case 3: $\gamma  = \alpha_1 + \alpha_2 + \alpha_3$} In this case, $P=Q$, so we must have $J' =  \{ \alpha_1, \alpha_2, \alpha_3, \alpha_1+\alpha_2, \alpha_2+\alpha_3 \}$. We depict this case in Figure~\ref{fig:Case3A3}.
Again we just need to check the regions of the hyperplane arrangement $\bigcup_{\beta \in J'} \beta^{\perp}$ which lie entirely on the negative side of $\gamma^{\perp}$. There are $6$ such regions (shaded in gray), with corresponding $B'$ sets $\{ \alpha_1, \alpha_2, \alpha_3, \alpha_1+\alpha_2, \alpha_2+\alpha_3 \}$, $\{ \alpha_2, \alpha_3, \alpha_1+\alpha_2, \alpha_2+\alpha_3 \}$, $\{ \alpha_1, \alpha_3, \alpha_1+\alpha_2, \alpha_2+\alpha_3 \}$, $\{ \alpha_1, \alpha_2, \alpha_1+\alpha_2, \alpha_2+\alpha_3 \}$, $\{  \alpha_1, \alpha_3, \alpha_1+\alpha_2 \}$ and $\{ \alpha_1, \alpha_3, \alpha_2+\alpha_3 \}$. In every one of these cases, either $\{ \alpha_1, \alpha_2+\alpha_3 \} \subseteq B'$, or $\{ \alpha_3, \alpha_1+\alpha_2 \} \subseteq B'$ (or both). So $\alpha_1+\alpha_2+\alpha_3$ is in the closure of $B'$ as required.

\begin{figure}[h]
	\begin{tikzpicture}
		\begin{scope}
			\path[clip] (2.25, 2.25) -- (-2.25, -2.25) -- (-2.25, 2.25) -- cycle;
			\fill[gray]  (0.5,0) circle (1.25) ; 
		\end{scope}
		\begin{scope}
			\path[clip] (-2.25, 2.25) -- (2.25, -2.25) -- (2.25, 2.25) -- cycle;
			\fill[gray]  (-0.5,0) circle (1.25) ; 
		\end{scope}
		\draw (2.25,2.25) -- (-2.25,-2.25) node[left] {$\alpha_1^{\perp}$};
		\draw (-2.25,2.25) -- (2.25,-2.25) node[right] {$\alpha_3^{\perp}$};
		\draw (0,-0.5) circle (1.25) node [xshift=20, yshift=-35] {}; %{$\alpha_2^{\perp}$};
		\draw (-0.5,0) circle (1.25) node [xshift=-65] {$(\alpha_1+\alpha_2)^{\perp}$};
		\draw (0.5,0) circle (1.25) node [xshift=65] {$(\alpha_2+\alpha_3)^{\perp}$};
		\draw[very thick] (0,0.5) circle (1.25) node [yshift=45] {$\gamma^{\perp}$};
		\draw (0, -2.1) node {$\boldsymbol{D}$};
	\end{tikzpicture}
	\caption{Case 3 in the proof of Lemma~\ref{lem:KeyCleanLemmaRephrase} for $A_3$.} \label{fig:Case3A3}
\end{figure}
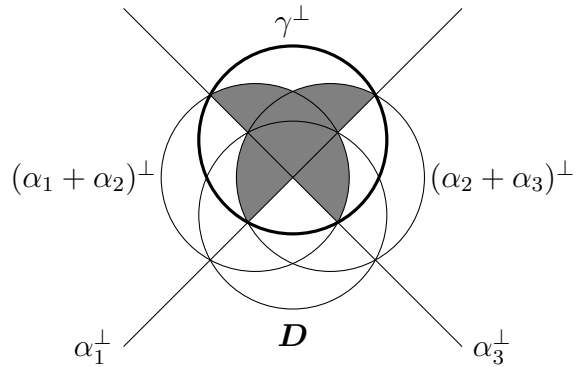

%\end{proof}

\section{Verification of  Lemma~\ref{lem:KeyCleanLemmaRephrase} in type $\tA_2$}\label{sec:A2TildeSuitable}
Let $\Phi$ be a root system of type $\tA_2$.
We write $\alpha_1$, $\alpha_2$, $\alpha_3$ for the fundamental roots and $\delta = \alpha_1+\alpha_2+\alpha_3$ for the imaginary root.
We define:
\[ 
%\begin{array}{r@{}c@{}l@{\quad}r@{}c@{}l@{\quad}r@{}c@{}l}
\beta_1^0 = \alpha_1 \quad  \beta_2^0 = \alpha_1+\alpha_2 \quad  \beta_3^0 = \alpha_2 \quad 
\beta_4^0 = \alpha_2+\alpha_3  \quad  \beta_5^0 = \alpha_3  \quad   \beta_6^0 = \alpha_1+\alpha_3. 
%\end{array}
\]
We define $\beta_i^k = \beta_i^0 + k \delta$.  The positive roots are $\beta_i^k$ for $1 \leq i \leq 6$ and $k \geq 0$.
We always take the subscripts on the $\beta$'s to be cyclic modulo $6$.
For each $1 \leq i \leq 3$, the roots $\beta_i^0$ and $\beta_{i+3}^0$ are the fundamental vectors of an $\tA_1$ subsystem, with $\beta_i^0+\beta_{i+3}^0 = \delta$.

Our goal in this section is to prove Lemma~\ref{lem:KeyCleanLemmaRephrase} in type $\tA_2$.
Thus, fix throughout this section a positive real root $\gamma = \beta_g^k$ and a finite order ideal $J$ in which $\gamma$ is maximal. 
For each $1 \leq h \leq 6$, we have $\beta_h^0 \prec \beta_h^1 \prec \beta_h^2 \cdots$, so there is some index $k_h \geq -1$ such that $J \cap \{ \beta_h^j : j \geq 0 \}$ is $\{ \beta_h^j : j \leq k_h \}$. 
We introduce the abbreviation $\beta_h^{\max}$ for $\beta_h^{k_h}$.
We also set $r=1$ if $g$ is odd and $r=0$ if $g$ is even. %(where $g$ is defined by the relation $\gamma = \beta_g^k$).
The following lemmas are immediate:

\begin{lemma}
	In the above notation, we have $\beta_{g-1}^a + \beta_{g+1}^b = \beta_g^{a+b+r}$. 
\end{lemma}

\begin{lemma} \label{lem:Comparable}
	For any indices $p$ and $q$ and any $j \geq 0$, we have $\beta_p^j \prec \beta_q^{j+1}$.
\end{lemma}

%\begin{proof}
%We temporarily extend the notation $\prec$ to a partial order on all integer combinations of roots, rather than just on real roots, by defining $\sum a_i \alpha_i \preceq \sum b_i \alpha_i$ if $a_i \leq b_i$ for $1 \leq i \leq 3$.
%Clearly, this restricts to the root order on roots.
%
%Since $\beta_p+\beta_{p+3} = \delta$, we have $\beta_p \prec \delta$. Therefore, 
%\[ \beta_p^k = \beta_p + k \delta \prec (k+1) \delta \prec \beta_q+(k+1) \delta = \beta_q^{k+1} . \qedhere \]
%\end{proof}
%

Recall that $k$ is the index such that $\gamma = \beta_g^k$. Lemma~\ref{lem:Comparable} immediately implies:
\begin{lemma}  \label{lem:FormOfIdeal}
	Each of the $k_h$ is either $k-1$ or $k$.
\end{lemma}

\begin{proof}
	We need to show that $\beta_h^{k-1} \in J$ and $\beta_h^{k+1} \not\in J$. For the first claim, Lemma~\ref{lem:Comparable} shows that $\beta_h^{k-1} \prec \beta_g^k = \gamma$, and $J$ is an order ideal containing $\gamma$.  For the second claim, Lemma~\ref{lem:Comparable} shows that $\beta_h^{k+1} \succ \beta_g^k = \gamma$, and $J$ is an order ideal in which $\gamma$ is maximal, so $J$ cannot contain any root which dominates $\gamma$.
\end{proof}

\begin{lemma} \label{lem:TightBound}
	In the above notation, we have $k_{g-1} = k_{g+1} = k-r$.
\end{lemma}

\begin{proof}
	First, suppose that $r=0$. Then $\beta_{g \pm 1}^k \prec \beta_{g}^k=\gamma$.  Since $\gamma \in J$ and $J$ is an order ideal, this shows that $\beta_{g \pm 1}^k \in J$ and $k_{g \pm 1} = k$.
	
	Now, suppose that $r=1$. Then $\beta_{g \pm 1}^k \succ \beta_{g}^k=\gamma$.  Since $\gamma$ is maximal in $J$, we deduce that $\beta_h^k \not\in J$, and thus $k_h = k-1$.
\end{proof}

Again, our goal in this section is to prove Lemma~\ref{lem:KeyCleanLemmaRephrase}. Write $J'=J\setminus\{\gamma\}$.
Fix throughout this section a region $\Omega$ of the $J'$-hyperplane arrangement. Our goal is to show that one of the following holds:
\begin{enumerate}
	\item There are $\zeta_1$ and $\zeta_2$ in $J'$ with $\gamma \in \Span_+(\zeta_1, \zeta_2)$ such that $\langle \zeta_1, \Omega \rangle$ and $\langle \zeta_2, \Omega \rangle$ have the same sign, or 
	\item $\gamma^{\perp}$ passes through the interior of $\Omega$.
\end{enumerate}
When we show that either of these hold, we will say that ``$\Omega$ is safe". So our goal is to show that all regions are safe.

We will depict our arguments visually, and we now explain the conventions with which we draw our diagrams.
Replacing $\Omega$ by $- \Omega$ if necessary, we may, and do, \textbf{assume that $\Omega$ meets the Tits cone $\{ \theta\in \Vd : \langle \delta, \theta \rangle > 0 \}$}.
Figure~\ref{fig:tA2} depicts the intersection of the $J$-hyperplane arrangement with the hyperplane $\{ \theta\in\Vd : \langle \delta, \theta \rangle =1 \}$.
We will use language that refers to the geometry of this diagram frequently, talking about ``parallel planes", ``rhombi", ``triangles", etcetera.

Our choice to use affine arrangements means that we can use the classical representation of $\tA_2$ as an affine reflection group.
However, we must point out one subtlety: $\Omega$ is safe if $\gamma^{\perp}$ passes through the interior of $\Omega$, but $\Omega$ may extend both above and below the plane $\delta^{\perp}$. 
Specifically, there are two regions of the hyperplane arrangement which extend both above and below $\delta^{\perp}$ but whose intersection with $\gamma^{\perp}$ is entirely in  on the negative side of $\delta^{\perp}$; we study these regions in Lemma~\ref{lem:RegionYPart2}.
$\Omega$ is safe in those cases even though our visual conventions mean that we can't see the hyperplane $\gamma^{\perp}$ meeting the region.

Here is the first easy case in which we know $\Omega$ is safe.
\begin{lemma} \label{lem:ParallelPlanes}
	The roots $\beta_g^{k-1}$ and $\beta_{g+3}^{\max}$ lie in $J'$. If $\Omega$ lies between the parallel hyperplanes $(\beta_g^{k-1})^{\perp}$ and $(\beta_{g+3}^{\max})^{\perp}$, then $\Omega$ is safe.
\end{lemma}

\begin{proof}
	The root $\beta_g^{k-1}$ lies in $J'$ since $\beta_g^{k-1} \prec \beta_g^k$, and the root $\beta_{g+3}^{\max}$ lies in $J'$ by definition, so we have verified the first sentence.
	The root $\beta_g^k$ is in the positive span of $\beta_g^{k-1}$ and $\beta_{g+3}^{\max}$. If $\Omega$ lies between these hyperplanes, then $\langle \beta_g^{k-1}, \Omega \rangle$ and $\langle \beta_{g+3}^{\max}, \Omega \rangle$ are both $>0$, so $\Omega$ is safe.
\end{proof}

Here is the other main case where $\Omega$ is safe:
\begin{lemma} \label{lem:EasyWedges}
	Let $0 \leq a \leq k-r$. Then $\beta_{g-1}^a$ and $\beta_{g+1}^{k-r-a}$ are both in $J'$; if $\langle \beta_{g-1}^a, \Omega \rangle$ and $\langle \beta_{g+1}^{k-r-a}, \Omega \rangle$ have the same sign, then $\Omega$ is safe.
\end{lemma}

\begin{proof}
	From Lemma~\ref{lem:TightBound}, we have $k_{g-1} = k_{g+1} = k-r$. Since $a \leq k-r$ and $k-r-a \leq k-r$, we deduce that $\beta_{g-1}^a$ and $\beta_{g+1}^{k-r-a} \in J$. We also have $\beta_g^k = \beta_{g-1}^a+\beta_{g+1}^{k-r-a}$. Thus, if $\langle \beta_{g-1}^a, \Omega \rangle$ and $\langle \beta_{g+1}^{k-r-a}, \Omega \rangle$ have the same sign, then $\Omega$ is safe.
\end{proof}

Define $K$ to be the following set of positive roots:
\[
K= \{ \beta_g^{k-1}, \beta_{g+3}^{\max} \} \cup \{ \beta_{g \pm 1}^j \ : \ 0 \leq j \leq k-r \} 
\]

Lemmas~\ref{lem:ParallelPlanes} and~\ref{lem:EasyWedges} show that $K \subseteq J'$, so the $J'$-hyperplane arrangement refines the $K$-hyperplane arrangment.
The $K$-hyperplane arrangement are the lines of ordinary thickness in Figure~\ref{fig:tA2}; the bold line is $\gamma^{\perp}$.
The dashed lines are hyperplanes which may be in $J'$ but are not in $K$.
%In the particular diagram, we have taken $g$ even, so $r=0$, with $k_g=k_{g-1}=k_{g+1}=3$ and $k_{g+3} = 2$.

There are many regions of the $K$-hyperplane arrangement such that, if $\Omega$ is one those regions, then Lemmas~\ref{lem:ParallelPlanes} and~\ref{lem:EasyWedges} tell us that $\Omega$ is safe immediately; those regions are shaded gray in Figure~\ref{fig:tA2}.
We have labeled the fundamental domain $D$, so $D$ is on the positive side of every Coxeter hyperplane.
The remainder of the proof is working through the remaining regions of Figure~\ref{fig:tA2} and checking that $\Omega$ is safe in those cases as well.
We have labeled these remaining regions $X_+$ and $X_-$ (blue), $R_1$ through $R_{k-r}$ (red), and $Y_+^1$ and $Y_-^1$ (green). 
These cases are addressed in Lemmas~\ref{lem:RegionX}, Lemma~\ref{lem:RegionR}, and Lemmas~\ref{lem:RegionYPart1} and~\ref{lem:RegionYPart2}, respectively.
%Note that the bold line $\gamma^{\perp}$, as well as the dashed lines, are not in $K$; the regions $X_{\pm}$ and $R_a$ extend to both sides of these lines.

\usetikzlibrary{math}

\begin{figure}[h]
	\begin{tikzpicture}[scale=0.7]
		%X regions
		\fill[cyan] (-10, 5*1.732) -- (-4,  5*1.732) -- (-2,1.732*2) -- (-2.5,1.732*1) -- (-10, 1.732*1) -- cycle; 
		\fill[cyan] (10, 5*1.732) -- (4,  5*1.732) -- (2,1.732*2) -- (2.5,1.732*1) -- (10, 1.732*1) -- cycle;
		%P regions
		\fill[red!80] (-2, 2*1.732) -- (-1,  3*1.732) -- (0,1.732*2) -- (-1,  1*1.732) -- cycle;
		\fill[red!80] (0, 2*1.732) -- (1,  3*1.732) -- (2,1.732*2) -- (1,  1*1.732) -- cycle;
		%Y regions
		\fill[green] (-10, -4*1.732) -- (-10, -2*1.732) -- (-6, -2*1.732) -- (-8, -4*1.732) -- cycle;
		\fill[green] (10, -4*1.732) -- (10, -2*1.732) -- (6, -2*1.732) -- (8, -4*1.732) -- cycle;
		%boring regions
		\fill[lightgray] (-10, -2*1.732) -- (10, -2*1.732) -- (10, 1*1.732) --  (-10, 1*1.732) -- cycle;
		\foreach \j in {0,1,...,2}   {
			\fill[lightgray, even odd rule] (5-2*\j, 5*1.732) -- (-4-2*\j, -4*1.732) --  (8-2*\j, -4*1.732) -- (-1-2*\j, 5*1.732) -- cycle;
		}  
		\draw[very thick] (-10,1.732*2) -- (10,1.732*2) node[above left] {$(\beta_g^{\max})^{\perp}=\gamma^{\perp} \!\!\!$};
		\draw (-10,1.732*1) -- (10,1.732*1) node[above  left] {$(\beta_g^{k_g-1})^{\perp}$};
		\draw (-10,-1.732*2) -- (10,-1.732*2) node[below  left] {$(\beta_{g+3}^{\max})^{\perp}$};
		\foreach \j in {0,1,...,2}  {
			\draw (5-2*\j, 5*1.732) -- (-4-2*\j, -4*1.732);
			\draw (-5+2*\j, 5*1.732) -- (4+2*\j, -4*1.732);
		}
		\foreach \j in {-2,...,-1}  {
			\draw[dashed] (5-2*\j, 5*1.732) -- (-4-2*\j, -4*1.732);
			\draw[dashed] (-5+2*\j, 5*1.732) -- (4+2*\j, -4*1.732);
		}
		\foreach \j in {-1,...,1} {
			\draw[dashed] (-10,1.732*\j) -- (10,1.732*\j);
		}
		\draw (0, -1.155) node {$\boldsymbol{D}$};
		
		\draw[cyan, fill=cyan] (-5,1.732*2) ellipse (0.5 and 0.5);
		\draw (-5,1.732*2) node {$X_+$};
		\draw[cyan, fill=cyan] (5,1.732*2) ellipse (0.5 and 0.5);
		\draw (5,1.732*2) node {$X_-$};
		%\draw (3+3,1.732*3+0.5) node {$X_-$};
		%\draw (-3-3,1.732*3+0.5) node {$X_+$};
		\draw[red!80, fill=red!80] (-1,1.732*2) ellipse (0.3 and 0.3);
		\draw (-1,1.732*2) node {$R_1$};
		\draw[red!80, fill=red!80] (1,1.732*2) ellipse (0.3 and 0.3);
		\draw (1,1.732*2) node {$R_2$};
		\draw (-8.5,-1.732*3) node {$Y^1_-$};
		\draw (8.5,-1.732*3) node {$Y^1_+$};
		\draw (-4, -4*1.732) node[below] {$(\beta_{g-1}^0)^{\perp}$};
		\draw (-8, -4*1.732) node[below] {$(\beta_{g-1}^{\max})^{\perp}$};
		\draw (4, -4*1.732) node[below] {$(\beta_{g+1}^0)^{\perp}$};
		\draw (8, -4*1.732) node[below] {$(\beta_{g+1}^{\max})^{\perp}$};
	\end{tikzpicture}
	\caption{The various regions in our proof of Lemma~\ref{lem:KeyCleanLemmaRephrase} in Type $\tA_2$.} \label{fig:tA2}
\end{figure}
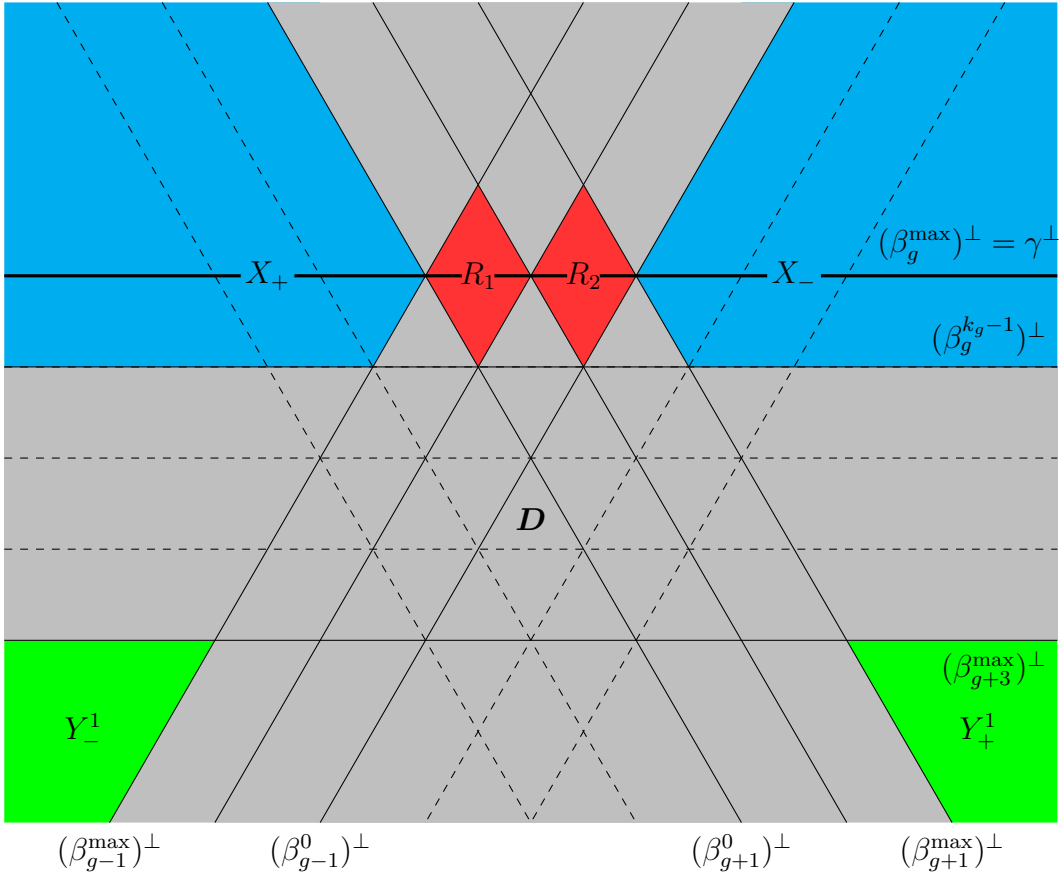

We first discuss the regions $X_{\pm}$ (blue).
\begin{lemma} \label{lem:RegionX}
	Let $X_{\pm}$ be the (unbounded) polyhedron cut by the following inequalities: 
	\[\langle \beta_{g}^{k-1}, -\rangle < 0,~ \langle \beta_{g \mp 1}^{\max}, -\rangle < 0,\text{ and }\langle \beta_{g \pm 1}^0, -\rangle > 0.\]
	If $\Omega \subseteq X_{\pm}$, then $\Omega$ is safe. 
\end{lemma}

\begin{proof}
	Without loss of generality, we take the $\pm$ sign to be $+$.
	
	The Coxeter hyperplanes crossing the interior of $X_+$ are dual to the roots $\{ \beta_{g-2}^j: j \geq 0 \}$, $\{ \beta_{g-1}^j: j>k_{g-1} \}$ and $\{ \beta_{g}^j: j> k-1 \}$.
	The latter two sets of roots are not in $J'$. 
	So the only $J'$ hyperplanes dividing up $X_+$ are $(\beta_{g-2}^j)^{\perp}$ for $0 \leq j \leq k_{g-2}$ (dashed in the figure). 
	These divide $X$ into parallel strips, and one cone with a $60^{\circ}$ angle,  and $\gamma^{\perp}$ passes through the interior of all of them, so $\Omega$ is safe.
\end{proof}

We next discuss the $k-r$ rhombi $R_a$ (red): 
\begin{lemma} \label{lem:RegionR}
	Let 
	\[R_a =  \{ \theta\in\Vd: \langle \beta_{g-1}^{a-1}, \theta \rangle > 0 > \langle \beta_{g-1}^a,\theta \rangle,\ \langle \beta_{g+1}^{k-r-a}, \theta \rangle > 0 > \langle \beta_{g+1}^{k-r-a+1}, \theta \rangle \} .\]
	for $1 \leq a \leq k-r$. If $\Omega=R_a$, %\subseteq P$,  not sure what the P was referring to
	then $\Omega$ is safe.
\end{lemma}

\begin{proof}
	The only Coxeter hyperplane which passes through the interior of $R_a$ is $\gamma^{\perp}$, so $\gamma$ passes through the interior of $\Omega$, as desired.
\end{proof}

Finally, we turn to the case where $\Omega$ lies in the region of the figure labeled $Y^1_{\pm}$ (green), bounded by $(\beta_{g+3}^{\max})^{\perp}$ and $(\beta_{g \pm 1}^{\max})^{\perp}$. 
This case is tricky to discuss, since it is the one case in which we need to think about points on the negative side of $\delta^{\perp}$. 
Thus, we need to carefully distinguish between $3$-dimensional cones, and their $2$-dimensional intersections with $\{ \theta : \langle \delta, \theta \rangle = 1 \}$. 
To this end, we make the following definitions. Let $\Omega$ be the three dimensional cone, in the central $J'$-hyperplane arrangement, for which we are trying to verify Lemma~\ref{lem:KeyCleanLemmaRephrase}. 
Let $\Omega^1 = \Omega \cap  \{ \theta : \langle \delta, \theta \rangle = 1 \}$. 

Fortunately, in this case, $Y^1_{\pm}$ and $\Omega^1$ are the same thing, as verified by the following lemma:

\begin{lemma} \label{lem:RegionYPart1}
	With the above notation,  there are no $J'$-hyperplanes meeting the interior of $Y_{\pm}^1$.
	Thus, if $\Omega^1$ is in $Y_{\pm}^1$, then $\Omega^1=Y^1_{\pm}$.
\end{lemma}

\begin{proof}
	Without loss of generality, we assume that the $\pm$ sign is $+$.
	The Coxeter hyperplanes that cross the interior of $Y_+^1$  are dual to the roots $\{ \beta_{g+1}^j: j>k_{g+1} \}$, $\{ \beta_{g+2}^j: j>\ell \}$ and  $\{ \beta_{g+3}^j: j>k_{g+3} \}$ where $\ell = k_{g+1}+k_{g+3}+r$. 
	The significance of the bound $\ell$ is that $\beta_{g+2}^{\ell} = \beta_{g+1}^{\max} + \beta_{g+3}^{\max}$, so that $(\beta_{g+2}^{\ell})^{\perp}$ passes through the corner $(\beta_{g+1}^{\max})^{\perp}  \cap (\beta_{g+3}^{\max})^{\perp}$ of $Y_+^1$.
	
	Clearly, the roots  $\{ \beta_{g+1}^j: j>k_{g+1} \}$ and $\{ \beta_{g+3}^j: j>k_{g+3} \}$ are not in $J'$.
	It remains to verify that $\ell \geq k_{g+2}$, in other words, that $k_{g+1} + k_{g+3} + r \geq k_{g+2}$.
	Since each of $k_{g+1}$, $k_{g+2}$, $k_{g+3}$ is either $k$ or $k-1$ and $r$ is either $0$ or $1$, this is immediate for $k \geq 2$.
	We leave the finitely many cases where $k \leq 1$ to the reader.
\end{proof}

\begin{lemma} \label{lem:RegionYPart2}
	If $\Omega^1 \subseteq Y_{\pm}^1$, then $\gamma^{\perp}$ meets the interior of $\Omega$, so $\Omega$ is safe.
\end{lemma}

\begin{proof}
	The cone of the $J'$-hyperplane arrangement containing $Y^1_{+}$ is bounded by $(\beta_{g+3}^{\max})^{\perp}$, $(\beta_{g+1}^{\max})^{\perp}$, $(\beta_g^{k_g-1})^{\perp}$ and $(\beta_{g-2}^{\max})^{\perp}$.
	So this cone must be $\Omega$.
	Let $\rho$ be the ray of $\Omega$ along the line  $(\beta_{g+3}^{\max})^{\perp} \cap (\beta_{g}^{k_{g-1}})^{\perp}$. The ray $\rho$ is in $\delta^{\perp}$, and hence corresponds to the point at infinity on the far right of Figure~\ref{fig:tA2}.
	In order to depict $\Omega$ more clearly, we slice the three dimensional hyperplane arrangement along an affine plane $H$ transverse to $\rho$, and use geometric language in the slice $H$. We depict the slice with $H$ in Figure~\ref{fig:tA2H}. The top half of the figure (shaded primarily in gray) is the Tits cone. The region $\Omega$ (green) extends into both the Tits cone and the negative Tits cone. To help orient the reader, we have also drawn $X_-$ (blue), which also extends into  the Tits cone and the negative Tits cone. 
	
	So $\Omega^1$ meets $H$ along a line segment, which we have drawn as a thick line, and $\Omega$ meets $H$ in a quadrilateral. The ray $\rho$ meets $H$ at a vertex of this quadrilateral.
	The hyperplane $\gamma^{\perp}=(\beta_g^{\max})^{\perp}$ passes through the ray $\rho$ of $\Omega$ and lies between the bounding hyperplanes  $(\beta_{g+3}^{\max})^{\perp}$ and  $(\beta_g^{k_g-1})^{\perp}$ of $\Omega$. 
	So $\gamma^{\perp}$ passes through the interior of $\Omega$, as promised. 
\end{proof}

\usetikzlibrary{math}

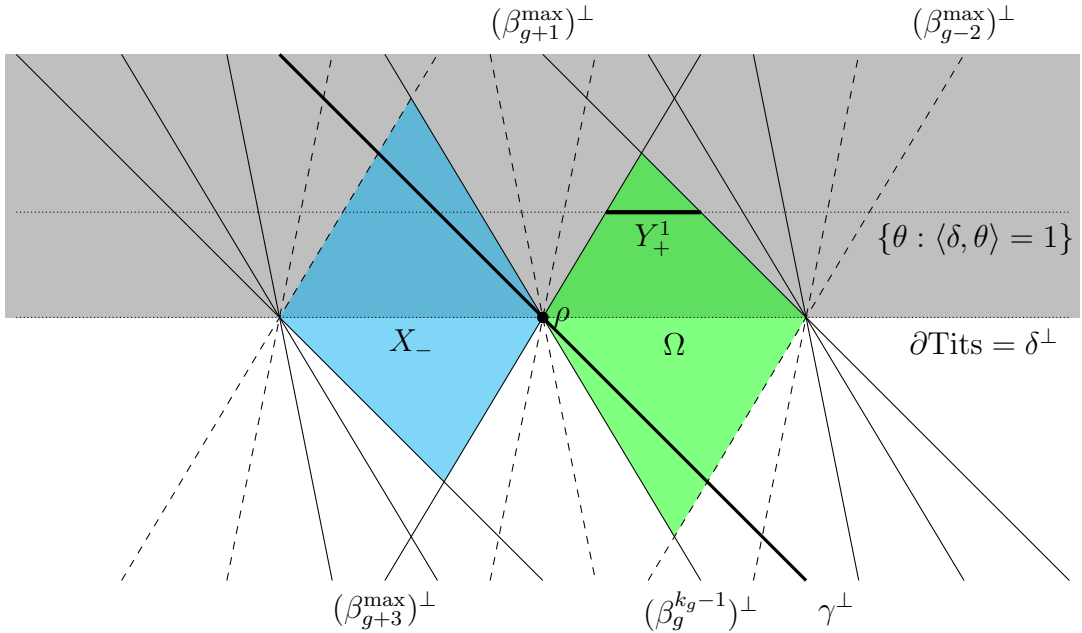
\begin{figure}[h]
	\begin{tikzpicture}[scale=0.7]
		\fill[lightgray] (-10.2,0) -- (-10.2,5) -- (10.2,5) -- (10.2,0) -- cycle;
		\fill[green, opacity=0.5] (0,0) -- (1.875,3.125) -- (5,0) -- (2.5,-4.167) -- cycle;
		\fill[cyan, opacity=0.5] (-5,0) -- (-2.5, 4.167) -- (0,0) -- (-1.875,-3.125) -- cycle;
		\draw[densely dotted] (-10,2) -- (10,2) node[below left] {$\{ \theta : \langle \delta, \theta \rangle = 1 \} \!\!\!$ };
		\draw[densely dotted] (-10,0) -- (10,0) node[below left] {$\partial \text{Tits} = \delta^{\perp}$};
		%\beta_g here
		%\draw (-1,-5) -- (1,5);
		\draw (3,5) -- (-3,-5)   node[below] {$(\beta_{g+3}^{\max})^{\perp}$};
		\draw[very thick] (-5,5) -- (5,-5) node[below right] {$\gamma^{\perp}$};
		\draw (-3,5) -- (3,-5)  node[below] {$(\beta_g^{k_g-1})^{\perp}$};
		\draw[dashed] (-1,5) -- (1,-5);
		\draw[dashed] (-1,-5) -- (1,5);
		%\beta_{g-1} here
		\draw[dashed] (-6,-5) -- (-4,5);
		\draw[dashed] (-8,-5) -- (-2,5);
		\draw (-10,5) -- (0,-5);
		\draw (-8,5) -- (-2,-5);
		\draw (-6,5) -- (-4,-5);
		%\beta_{g+1} here
		\draw[dashed] (4,-5) -- (6,5);
		\draw[dashed] (2,-5) -- (8,5) node[above] {$(\beta_{g-2}^{\max})^{\perp}$};
		\draw (10,-5) -- (0,5) node[above] {$(\beta_{g+1}^{\max})^{\perp}$};
		\draw (2,5) -- (8,-5);
		\draw (4,5) -- (6,-5) ;
		\draw[ultra thick] (1.2,2) -- (3,2);
		\draw (2.5,-0.5) node {$\Omega$};
		\draw (-2.5,-0.5) node {$X_-$};
		\draw (2.1,1.5) node {$Y^1_+$};
		\draw[fill=black] (0,0) ellipse (0.1 and 0.1) node[right] {$\rho$};
	\end{tikzpicture}
	\caption{A different slice through the $\tA_2$ hyperplane arrangement.} \label{fig:tA2H} 
\end{figure}

\section{Folding} \label{sec:Folding} 
Let $(\Phi,\Pi,V)$ and $(\uPhi,\uPi,\uV)$ be the geometric realizations of root systems $\Phi$ and $\uPhi$. Assume we are given linear maps $V\xrightarrow{i} \uV \xrightarrow{p} V$ such that %$p\circ i = \mathrm{id}_{V}$. actually we don't need this
$i$ is injective and $p$ is surjective.
We can think of $V$ as a subspace of $\uV$, and $p$ as a projection of $\uV$ onto that subspace. (We do not require that $p\circ i = \mathrm{id}_V$, though this will be the case for our examples.) We ask for the following conditions:

\vspace{.5em}
(\textbf{Condition 1}) The image $p(\uPi)$ is $\Pi$. 

(\textbf{Condition 2}) The image $p(\uPhi)$ is $\Phi$.

\vspace{.5em}
\noindent When these first two conditions are satisfied, we will write $f: \uPhi \to \Phi$ for the restriction of $p$ to $\uPhi$. Note that in this case $f(\uPhi_+) = \pos$.

\vspace{.5em}
(\textbf{Condition 3}) If $\alpha\in \Phi$, then $i(\alpha)$ is contained in the nonnegative span $\Spanp f^{-1}\{\alpha\}$.

\vspace{.5em}
Examples of maps satisfying the conditions above come naturally from \newword{folding} Dynkin diagrams. Namely, let $(\uPhi,\uPi,\uV)$ be any geometrically realized root system. Choose a permutation $\sigma$ of the base $\uPi$ which preserves the pairings $(\alpha_i, \alpha_j)$, and assume that within each orbit of $\sigma$ on $\uPi$, the simple roots are pairwise orthogonal. Let the order of $\sigma$ be $m$. Then define $V$ to be the subspace of $\uV$ fixed by $\sigma$, and define the retraction $p: \uV \to V$ by
\[ p(\alpha) \coloneqq \frac{1}{m}\sum_{k=0}^{m-1} \sigma^k(\alpha). \]
We endow $V$ with the bilinear form which is $m$ times the restriction of the form on $\uV$. We define $\Pi$ to be the image of $\uPi$ under $p$. It turns out that $\Pi$ is the base of a unique root system $\Phi$ geometrically realized in $V$. %what if non-reduced? --> SL2 module rep theory implies it is reduced
If we take $i$ to be the inclusion $V\hookrightarrow \uV$, then the maps $i,p$ and the root systems $\Phi,\uPhi$ \textit{almost} satisfy the three conditions above.
For a general symmetry of a Dynkin diagram,  the averaging map above will satisfy Conditions 1 and 3, but only a weaker form of Condition 2: $p(\uPhi) \supseteq \Phi$. 

\begin{remark}
	For the expert reader, the trouble is that $f$ may send real roots of $\uPhi$ to imaginary roots of $\Phi$, which, recall, are not elements of $\Phi$ in our setup. For the non-expert, see \cite{StembridgeNote} for more details (though note that the map used for $p$ there is our $i^{*}$).
	
	As an example, let $\uPhi$ have type $\tA_5$, so the extended Dynkin diagram is a hexagon, and let $\sigma$ be the rotation of that hexagon by three positions.
	Then $\Phi$ will have extended Dynkin diagram a triangle, and thus should have type $\tA_2$. Since this is a folding, the induced maps $i,p$ satisfy Conditions 1 and 3.
	However, $\ualpha_1+\ualpha_2+\ualpha_3$ is a real root in $\uPhi$ and $f(\ualpha_1+\ualpha_2+\ualpha_3)$ is the imaginary root of $\Phi$, so Condition 2 is not satisfied.
\end{remark} 

There are two cases of interest where Condition 2 is always satisfied:
\begin{itemize}
	\item If $\uPhi$ is finite, or
	\item If $\uPhi$ is untwisted affine, and $\sigma$ fixes $\alpha_0$. In this case $\Phi$ will also be untwisted affine.
\end{itemize}

The foldings we will consider all fall into one of these two cases. We now list the specific foldings which we need by depicting the induced map $\uPi \to \Pi$ as a map between Dynkin diagrams:

\tikzset{/Dynkin diagram/fold style/.style={opacity=0}}
\tikzset{/Dynkin diagram/fold radius = .35cm}
\begin{figure}[H]
	\begin{tabular}{|@{\hskip2\tabcolsep}c c@{\hskip2\tabcolsep}|@{\hskip2\tabcolsep}c c@{\hskip2\tabcolsep}|@{\hskip2\tabcolsep}c c@{\hskip2\tabcolsep}|@{\hskip2\tabcolsep}c c@{\hskip2\tabcolsep}|}
		\hline
		$D_4$ & \dynkin[fold]{D}{4} & $A_5$ & \dynkin[fold]{A}{5} & $\tA_3$ & \dynkin[fold]{A}[1]{3} & $\tD_4$ & \dynkin[ply=3]{D}[1]{4} \\[.4cm]
		& \begin{tikzpicture}
			\draw[->] (0,.5) edge (0,0) (.35cm, .5) edge (.35cm, 0) (.7cm, .5) -- (.7cm, 0);
		\end{tikzpicture}\vspace{-.2cm} & &
		\begin{tikzpicture}
			\draw[->] (0,.5) edge (0,0) (.35cm, .5) edge (.35cm, 0) (.7cm, .5) -- (.7cm, 0);
		\end{tikzpicture} & &\begin{tikzpicture}
			\draw[->] (0,.5) edge (0,0) (.35cm, .5) edge (.35cm, 0) (.7cm, .5) -- (.7cm, 0);
		\end{tikzpicture} & &\begin{tikzpicture}
			\draw[->] (0,.5) edge (0,0) (.35cm, .5) edge (.35cm, 0) (.7cm, .5) -- (.7cm, 0);
		\end{tikzpicture} \\
		$B_3$ & \dynkin{B}{3} & $C_3$ & \dynkin{C}{3} & $\tC_2$ & \dynkin{C}[1]{2} & $\tG_2$ & \dynkin{G}[1]{2}  \\
		\hline
	\end{tabular}
	\caption{The foldings $D_4 \to B_3,\ A_5 \to C_3,\ \tA_3 \to \tC_2,$ and $\tD_4 \to \tG_2$, respectively.}
	\label{fig:folding}
\end{figure}
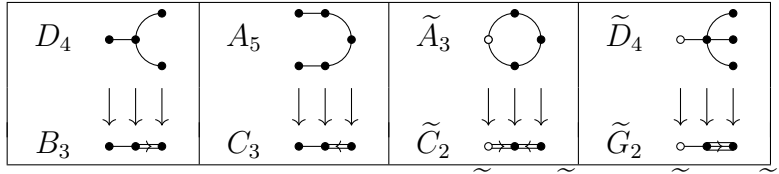
In each case, we leave it as an exercise for the reader to check that the conditions are met.

With these remarks and definitions out of the way, let $i$, $p$, and $f$ satisfy Conditions 1-3. For instance, we may take a folding from the list above.
We start with some lemmas:
\begin{lemma} \label{lem:PullbackOrderIdeal}
	Let $I$ be an order ideal in the root poset of $\pos$. Then $f^{-1}(I)$ is an order ideal in the root poset of $\uPhi_+$.
\end{lemma}

\begin{proof}
	Let $\ubeta \in f^{-1}(I)$ and let $\ugamma$ be another root of $\uPhi^+$ with $\ubeta \succ \ugamma$.
	We need to show that $\ugamma \in f^{-1}(I)$.
	Put $\beta= f(\ubeta)$ and $\gamma = f(\ugamma)$.
	
	The condition $\ubeta \succ \ugamma$ means that $\ubeta - \ugamma$ is a non-negative combination of the simple roots $\ualpha_i$. 
	Applying the linear map $p$, and applying Conditions 1 and 2, we see that $\beta - \gamma$ is a non-negative combination of the simple roots $\alpha_i$, so $\beta \succ \gamma$.
	Since $I$ is an order ideal, we deduce that $\gamma \in I$. 
	Then $\ugamma \in f^{-1}(I)$, as required.
\end{proof}

\begin{lemma} \label{lem:PullbackClosed}
	Let $I$ be an order ideal in the root poset of $\pos$ and let $B$ be biclosed in $I$. Then $f^{-1}(B)$ is biclosed in $f^{-1}(I)$.
\end{lemma}

\begin{proof}
	We will show that $f^{-1}(B)$ is closed in $f^{-1}(I)$; applying the same logic to $I \setminus B$ will show that $f^{-1}(B)$  is co-closed as well.
	Let $\ualpha$ and $\ubeta \in f^{-1}(B)$ and $\ugamma \in f^{-1}(I)$ with $\ugamma$ in the positive space of $\ualpha$ and $\ubeta$; we must show that $\ugamma \in f^{-1}(I)$.
	Set $f(\ualpha) = \alpha$, $f(\ubeta) = \beta$ and $f(\ugamma) = \gamma$.
	Using Condition 2, we have $\alpha$ and $\beta \in B$ and $\gamma \in I$.
	Applying the linear map $p$, we see that $\gamma$ is in the positive span of $\alpha$ and $\beta$. Since $B$ is biclosed in $I$, we deduce that $\gamma \in B$, and thus $\ugamma \in f^{-1}(B)$, as desired.
\end{proof}

\begin{lemma} \label{lem:PushforwardSeparable}
	Let $I$ be an order ideal in the root poset of $\pos$, let $B$ be biclosed in $I$ and suppose that $f^{-1}(B)$ is separable (respectively, weakly separable) in $f^{-1}(I)$. Then $B$ is separable (respectively, weakly separable) in $I$.
\end{lemma}

\begin{proof}
	We first discuss the relationship between separability and weak separability. If $I$ is finite, then separability and weak separability are the same thing. A general order ideal $I$ can be written as the rising union of finite order ideals: $I = \bigcup I_k$. If $B \cap I_k$ is separable in every $I_k$, then $B$ is weakly separable in $I$. Thus, it is enough to prove the version of the statement with separability. So, assume from now on that $f^{-1}(B)$ is separable  in $f^{-1}(I)$. This means that there is some $\utheta \in \uV^{\ast}$ such that $\langle \utheta, -\rangle$ is negative on $f^{-1}(B)$ and positive on $f^{-1}(I) \setminus f^{-1}(B)$.
	
	%Now, because the fibers of $f$ are orbits for $\sigma$, the linear map $\sigma : \uV \to \uV$ carries $f^{-1}(B)$ to itself, and carries $f^{-1}(I) \setminus f^{-1}(B)$ to itself.
	%Thus, $\sigma^{\ast}(\theta)$ also separates $f^{-1}(B)$ from $f^{-1}(I) \setminus f^{-1}(B)$, and so does $(\sigma^{\ast})^k(\theta)$ for any integer $k$.
	%Let $m$ be the order of the automorphism $\sigma$. We see that the average $\widehat{\phi} := \tfrac{1}{m} \sum_{k=0}^{m-1} (\sigma^{\ast})^k(\theta)$ likewise separates $f^{-1}(B)$ from $f^{-1}(I) \setminus f^{-1}(B)$.
	%
	%Now, $\widehat{\phi}$ is fixed by $\sigma$. The fixed locus of $\sigma$ is the image of the adjoint map $f^{\ast} : V^{\ast} \to \uV^{\ast}$, so there is some $\phi \in V^{\ast}$ with $f^{\ast}(\phi) = \widehat{\phi}$. Then $\phi$ separates $B$ from $I \setminus B$.
	
	Define $\theta\coloneqq i^{\ast}(\utheta)$. We claim that $f^{\ast}(\theta)$ also separates $f^{-1}(B)$ from $f^{-1}(I) \setminus f^{-1}(B)$. Indeed, for any $\ualpha\in \uPhi$, we have
	\[ \langle f^{\ast}(\theta), \ualpha\rangle = \langle f^{\ast}\circ i^{\ast}(\utheta), \ualpha\rangle = \langle \utheta, i\circ f(\ualpha)\rangle.  \]
	Applying Condition 3, we conclude that $i\circ f(\ualpha)$ is in the non-negative span of $f^{-1}\{ f(\ualpha) \}$. A root $\ubeta\in f^{-1}\{f(\ualpha)\}$ has $\langle \utheta, \ubeta\rangle < 0$ if and only if $\ubeta$ is in $f^{-1}(B)$ if and only if $\ualpha$ is in $f^{-1}(B)$. By expanding $i\circ f(\ualpha)$ as a non-negative combination of the $\ubeta$'s, we find that $\langle f^{\ast}(\theta),\ualpha\rangle < 0$ if and only if $\ualpha \in f^{-1}(B)$. It follows that $\langle \theta, f(\ualpha) \rangle < 0$ if and only if $f(\ualpha) \in B$. Hence, $\theta$ separates $B$.
\end{proof}

\section{Conclusion of the proof}
\label{sec:FoldedTypes}
We now prove Lemma~\ref{lem:KeyCleanLemma} in types $B_3$, $C_3$, $\tC_2$ and $\tG_2$. 
%We will then be able to deduce Theorems~\hyperlink{thm:introlattice}{A}, \hyperlink{thm:introclean}{B} and \hyperlink{thm:introextend}{C} in all finite and affine types, as we have already explained.

\begin{prop} \label{prop:foldedRank3}
	Let $\Phi$ be a root system of type $B_3$, $C_3$, $\tC_2$ or $\tG_2$. Let $\uPhi$ be the root system of type $D_4$, $A_5$, $\tA_3$ or $\tD_4$, respectively, and let $f: \uPhi \to \Phi$ be the folding as shown in \Cref{fig:folding}.
	Let $I$ be an order ideal in $\pos$ and let $B$ be biclosed in $I$.
	Then $B$ is separable in $I$.
\end{prop}

\begin{proof}
	From Lemmas~\ref{lem:PullbackOrderIdeal} and~\ref{lem:PullbackClosed}, $f^{-1}(I)$ is an order ideal in $\uPhi^+$ and $f^{-1}(B)$ is biclosed in $f^{-1}(I)$.
	Since $\uPhi$ is simply laced, we have proved Theorem~\hyperlink{thm:introextend}{C} for $\uPhi$. Thus, letting $\overline{B}$ be the closure of $f^{-1}(B)$ in $\uPhi$, we know that $\overline{B}$ is biclosed and $\overline{B}\cap f^{-1}(I) = f^{-1}(B)$.
	Moreover, since the closure operator defining $\overline{B}$ respects the symmetry $\sigma$ of $\uPhi$, we have $\sigma(\overline{B}) = \overline{B}$.
	
	If $\uPhi$ is $D_4$ or $A_5$, then $\overline{B}$ is weakly separable, since those positive root systems are clean. Otherwise, $\uPhi$ is affine. In~\cite{Barkley2022}, the authors characterized all biclosed sets in affine root systems and, in \cite[Section 5]{Barkley2022}, we determined which of those biclosed sets are weakly separable. In particular, to a biclosed set $\overline{B}$ of $\tA_3$ (respectively, $\tD_4$) there is an associated biclosed set $\overline{B}_\infty$ containing positive and negative roots of $A_3$ (respectively, $D_4$). The set $\overline{B}$ is not weakly separable if and only if there are roots $\alpha,\beta$ in $A_3$ (resp., $D_4$) such that $\{\pm \alpha\}\subseteq \overline{B}_\infty$ and $\{\pm \beta\}\cap \overline{B}_\infty = \varnothing$. One can check that this does not happen for $\sigma$-invariant biclosed subsets of the positive and negative roots of $A_3$ or $D_4$. Hence, any $\sigma$-invariant biclosed set in $\tA_3$ or $\tD_4$ is weakly separable.
	
	As a result, for all four of our foldings, any $\sigma$-invariant biclosed set is weakly separable.
	So we deduce that $\overline{B}$ is weakly separable in $\uPhi$, and therefore $f^{-1}(B)$ is separable in $f^{-1}(I)$ (which is finite).
	Then, by Lemma~\ref{lem:PushforwardSeparable}, $B$ is separable in $I$. 
\end{proof}

We have now proven Theorem~\hyperlink{thm:introclean}{B} for all finite root systems of rank $3$ and for untwisted affine root systems of rank $3$.
We therefore deduce Theorem~\ref{thm:NonSimplyLacedOrdering}: In any finite root system, or any untwisted affine root system, any total order refining the root poset is suitable.
Theorem~\ref{thm:BiclosedLattice} then implies Theorem~\hyperlink{thm:introlattice}{A}, 
Theorem~\ref{thm:BiclosedExtend} implies Theorem~\hyperlink{thm:introextend}{C}, 
and Theorem~\ref{thm:ConjectureA} implies Theorem~\hyperlink{thm:introConjD}{D} (Dyer's Conjecture~A).
What remains is to verify Theorem~\hyperlink{thm:introclean}{B} for finite root systems of rank $> 3$. We do that now:

\begin{theorem} \label{thm:FiniteClean}
	Let $\Phi$ be a finite root system. Let $J$ be any order ideal in the root order on $\pos$. Then $J$ is clean.
\end{theorem}

\begin{proof}
	Let $C$ be a biclosed set in $J$. We wish to show that $C$ is separable in $J$. By \Cref{thm:BiclosedExtend}, there is a biclosed set $B$ in $\pos$ such that $B\cap J = C$. Now we use that biclosed sets in $\pos$ are separable (e.g. by \cite{McConville2014}) to conclude that $B$ is separable in $\pos$. But this implies that $C$ is separable in $J$, so we are done.
\end{proof}

We remark on the connections between Theorem~\ref{thm:BiclosedLattice}, Theorem~\ref{thm:FiniteClean}, and results of Nathan Reading on the poset of regions of a Coxeter arrangement. 
Let $\Phi$ be a finite root system and let $J$ be an order ideal in the root order on $\pos$.
Theorem~\ref{thm:BiclosedLattice} tells us that biclosed sets in $J$ form a lattice.
It follows quickly from Reading's results \cite[Section 9-8]{Reading2016Ch9} that separable sets in $J$ form a lattice (in fact, a lattice quotient of the lattice of separable sets in $\pos$).
Theorem~\ref{thm:FiniteClean} ties these results together, showing that biclosed sets and separable sets are the same thing in $J$.

We have now completed our primary results. We conclude with comments on twisted affine root systems, and on the non-crystallographic types.

\section{Difficulties in twisted affine root systems} \label{sec:TwistedAffine}

Let $X \subset V$. If we replace any vector in $X$ by a positive multiple of itself, this will not change whether or not any given subset of $X$ is biclosed and/or separable.
Likewise, it will not change whether or not $X$ is clean.

However, if $\Phi$ is a crystallographic root system, there can be another crystallographic root system $\Phi^t$ obtained by replacing some vectors in $\Phi$ by positive multiples of themself. 
We call $\Phi^t$ a \newword{twist} of $\Phi$. The root posets on $\Phi$ and $\Phi^t$ will generally be different, and thus they will have different order ideals.
For example, $B_3$ and $C_3$ are twists of each other, which is why we needed to verify separately that order ideals are clean in both $B_3$ and $C_3$.

In this section, we will discuss a twisted version of $\tC_2$ for which some order ideals are not clean, and show how this twist does not have the nice properties that we have proved for the untwisted $\tC_2$. 
This limits the extent to which the arguments in Section~\ref{sec:FoldedTypes} can hope to be generalized to other foldings.

The root system we consider is denoted $D_3^{(2)}$ and has the extended Dynkin diagram \dynkin[label,verticalshift]{D}[2]{3}.
%\begin{center}
%	$D_3^{(2)}$\qquad \dynkin[label,verticalshift]{D}[2]{3}.
%\end{center} 
%The (perhaps misleading) name for this system comes from the existence of a folding $D_3=A_3 \to B_2$, where $B_2$ is the Dynkin diagram remaining after we delete the extra node $\alpha_0$. 
The (perhaps misleading) name comes from Kac's classification of affine root systems \cite{Kac1983}. We can construct it from $\tC_2$ as follows: the roots in $\tC_2$ all have length either $\sqrt{2}$ or $2$. This partitions $\tC_2$ into two root subsystems $\Phi_1$ and $\Phi_2$, respectively. (This is not a partition into irreducible components, since $\Phi_1$ and $\Phi_2$ are not orthogonal.) To obtain $D_3^{(2)}$, we rescale the elements of $\Phi_1$ to have length $2$, and rescale the elements of $\Phi_2$ to have length $\sqrt{2}$. This gives a bijection $\mathrm{tw}:\tC_2 \to D_3^{(2)}$. Write $\alpha_0,\alpha_1,\alpha_2$ for the simple roots of $\tC_2$ and $\beta_0\coloneqq \frac{1}{\sqrt{2}}\alpha_0, \beta_1\coloneqq \sqrt{2}\alpha_1, \beta_2\coloneqq \frac{1}{\sqrt{2}}\alpha_2$ for the simple roots of $D_3^{(2)}$. The following list of nine roots is an order ideal $I$ in the root poset of $D_3^{(2)}$. We give both the subset of $D_3^{(2)}$ and its preimage under $\mathrm{tw}$.

\begin{figure}[H]
	\begin{tabular}{c@{\hskip12\tabcolsep}c}
		$I\subseteq D_3^{(2)}$  & $\mathrm{tw}^{-1}(I)\subseteq \tC_2$  \\
		\rowcolor{gray!25}
		$\beta_0$ & $\alpha_0$ \\
		\rowcolor{gray!25}
		$\beta_1$ & $\alpha_1$ \\
		$\beta_2$ & $\alpha_2$ \\
		\rowcolor{gray!25}
		$\beta_0+\beta_1$ & $\alpha_0 + 2\alpha_1$ \\
		$\beta_1+\beta_2$ & $2\alpha_1 + \alpha_2$ \\
		\rowcolor{gray!25}
		$2\beta_0+\beta_1$ & $\alpha_0+\alpha_1$ \\
		$\beta_1+2\beta_2$ & $\alpha_1+\alpha_2$\\
		$2\beta_0+\beta_1+\beta_2$ & $2\alpha_0 + 2\alpha_1 + \alpha_2$ \\
		\rowcolor{gray!25}
		$\beta_0+\beta_1+2\beta_2$ & $\alpha_0 + 2\alpha_1 + 2\alpha_2$
	\end{tabular}
\end{figure}

We remark that $\mathrm{tw}^{-1}(I)$ is \emph{not} an order ideal of $\tC_2$: the root $\alpha_0+\alpha_1+\alpha_2$ is less than $2\alpha_0+ 2\alpha_1 + \alpha_2$ but is not in $\mathrm{tw}^{-1}(I)$. So twisting does not preserve order ideals. One might hope that Theorem~\hyperlink{thm:introclean}{B} is still true for $D_3^{(2)}$. This would mean that $I$ is a clean set. But this is not the case; the following set is biclosed in $I$ and yet not separable in $I$:
\[ B = \{ \beta_0, \beta_1, \beta_0+\beta_1, 2\beta_0+\beta_1, \beta_0+\beta_1+2\beta_2 \}. \]
We have indicated the elements of $B$ with shading in the table above. One can see that $B$ is not separable by presuming that $\theta\in \Vd$ is a separating vector and examining \[\langle\theta,2\beta_0+\beta_1+2\beta_2\rangle.\] 
This pairing must be negative, since $\beta_0$ and $\beta_0+\beta_1+2\beta_2$ are both in $B$ and hence their sum pairs negatively. The pairing must also be positive, since $2\beta_0+\beta_1+\beta_2$ and $\beta_2$ are both in $I\setminus B$ and hence their sum pairs positively. This is a contradiction, so $B$ is not separable in $I$.

This also provides a counterexample to Theorem~\hyperlink{thm:introextend}{C} in $D_3^{(2)}$. The biclosed set $B$ in $I$ cannot be extended to a biclosed set in $I \cup \{ 2\beta_0+\beta_1+2\beta_2 \}$, and thus cannot be extended to a biclosed set in $\pos$.

If one didn't know about this counterexample, one might hope to use the folding
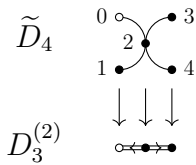
\begin{center}
	\begin{tabular}{c c}
		$\tD_4$ & \dynkin[fold,label]{D}[1]{4} \\[.4cm]
		& \begin{tikzpicture}
			\draw[->] (0,.5) edge (0,0) (.35cm, .5) edge (.35cm, 0) (.7cm, .5) -- (.7cm, 0);
		\end{tikzpicture}\vspace{-.2cm} \\
		$D_3^{(2)}$ & \dynkin{D}[2]{3} \\
	\end{tabular}
\end{center} 
and the argument in Section~\ref{sec:FoldedTypes} to prove that $I$ is clean. However, it turns out that this folding fails Condition 2 in Section~\ref{sec:Folding}. Indeed, the root $\ubeta_0+\ubeta_2+\ubeta_3$ of $\tD_4$ is sent to $\delta=\beta_0+\beta_1+\beta_2$ by the fold map. This vector $\delta$ is the primitive imaginary root of $D_3^{(2)}$, and hence \emph{not} in the root system under our convention.

The example in this section of an order ideal which is not clean is due to Matthew Dyer and was communicated to us by Thomas McConville.

\section{Remarks on $H_3$ and $H_4$} \label{sec:TypeH}
Recall that there are non-crystallographic finite Coxeter groups $(W,S)$, which do not admit root systems in the sense we use in this paper. However, such groups do admit faithful representations wherein reflections act via Euclidean reflection over a hyperplane. Generally, a ``root system'' in this setting is defined to be a set of normal vectors for these hyperplanes which is preserved by the action of $W$. These non-crystallographic root systems $\Phi$ still decompose into positive roots $\pos$ and negative roots $\Phi_-$. The biclosed sets in $\pos$ allow us to define an extended weak order for these groups, which coincides with the usual weak order since Coxeter arrangements are clean arrangements. 

The non-crystallographic finite Coxeter groups are all either of rank 2, or else are the rank 3 group $H_3$ or the rank 4 group $H_4$. One could ask: to what extent do the theorems presented here apply to the non-crystallographic root systems associated to these groups? If there is a suitable order on these root systems, then all the theorems from Sections~\ref{sec:extend} and \ref{sec:biclosedlattice} remain true. Any ordering on a rank 2 system putting the fundamental roots first will be suitable, so those systems are fine. But non-crystallographic root systems do not have a well-behaved notion of root poset%
\footnote{There is a Coxeter-theoretic notion of root poset \cite[Section 4.6]{Bjorner2005} which applies to these systems, but this is a different order than the one discussed here, and its order ideals can fail to be clean.}, 
which is a major obstacle to constructing a suitable order on $H_3$ or $H_4$ using the methods of this paper.

The methods in this paper can be adapted to prove a rather weak result in the cases of $H_3$ and $H_4$, which we now explain.
Let $\Phi$ be a root system of type $H_n$ for $n=3$ or $4$, with simple roots $\alpha_1$, $\alpha_2$, \dots, $\alpha_n$. 
We define a preorder $\preceq$ on $\pos$ as follows: Let $\sum p_i \alpha_i$ and $\sum q_i \alpha_i$ be positive roots. Then $\sum p_i \alpha_i \preceq \sum q_i \alpha_i$ if, for each index where $q_i=0$, we also have $p_i=0$. 
This is a preorder, meaning that it is reflexive and transitive, but not anti-symmetric. In any preorder, the relation defined by $\beta \sim \gamma$ if $\beta \preceq \gamma \preceq \beta$ is an equivalence relation, and the preorder induces a partial order on the equivalence classes of this equivalence relation.
Figure~\ref{fig:H4root} shows these partial orders for $H_3$ and $H_4$; the nodes of the Hasse diagram are labeled with the sizes of the equivalence classes. 

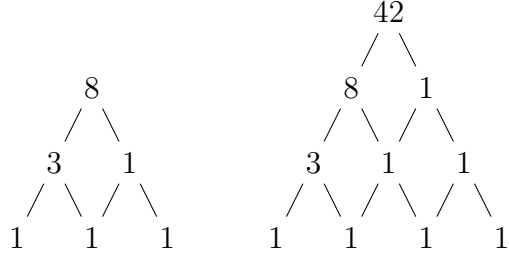
\begin{figure}[h]
	\centerline{
		\begin{tikzpicture}
			\node at (2, 2)   (b1) {$8$};
			\node at (1.5, 1)   (c1) {$3$};
			\node at (2.5, 1)   (c2) {$1$};
			\node at (1, 0)   (d1) {$1$};        
			\node at (2, 0)   (d2) {$1$};        
			\node at (3, 0)   (d3) {$1$};        
			\draw (b1) -- (c1); \draw (b1) -- (c2); 
			\draw (c1) -- (d1); \draw (c1) -- (d2); \draw (c2) -- (d2); \draw (c2) -- (d3); 
		\end{tikzpicture}
		\qquad
		\begin{tikzpicture}
			\node at (2.5, 3)   (a) {$42$};
			\node at (2, 2)   (b1) {$8$};
			\node at (3, 2)   (b2) {$1$};
			\node at (1.5, 1)   (c1) {$3$};
			\node at (2.5, 1)   (c2) {$1$};
			\node at (3.5, 1)   (c3) {$1$};
			\node at (1, 0)   (d1) {$1$};        
			\node at (2, 0)   (d2) {$1$};        
			\node at (3, 0)   (d3) {$1$};        
			\node at (4, 0)   (d4) {$1$};        
			\draw (a) -- (b1); \draw (a) -- (b2);
			\draw (b1) -- (c1); \draw (b1) -- (c2); \draw (b2) -- (c2); \draw (b2) -- (c3);
			\draw (c1) -- (d1); \draw (c1) -- (d2); \draw (c2) -- (d2); \draw (c2) -- (d3);  \draw (c3) -- (d3); \draw (c3) -- (d4);               
	\end{tikzpicture}}
	\caption{The preorders coming from the ``foldings'' $D_6 \to H_3$ and $E_8 \to H_4$.} \label{fig:H4root}
\end{figure}

We define an \newword{order ideal} in $\pos$ to be a subset $I$ of $\pos$ such that, if $\gamma \in I$, and $\beta \preceq \gamma$, then $\beta \in I$.
Note that these are unions of $\sim$-equivalence classes, forming order ideals in the quotient poset.

\begin{theorem}
	Let $I$ be a $\preceq$-order ideal. Then $I$ is clean.
\end{theorem}

\begin{proof}
	Let $(\uPhi, \uV)$ be a root system of type $D_6$ or $E_8$, according to whether $\Phi$ has type $H_3$ or $H_4$. We shall briefly describe a construction analogous to folding, which sends $D_6\to H_3$ and $E_8\to H_4$ according to the following diagrams.
	
	\vspace{.3em}\begin{figure}[H]
		\begin{tabular}{|@{\hskip2\tabcolsep}c c@{\hskip2\tabcolsep}|@{\hskip2\tabcolsep}c c@{\hskip2\tabcolsep}|}
			\hline
			$D_6$ & \begin{dynkinDiagram}[name=upper]A3
				\node (current) at ($(upper root 1)+(0,-.3cm)$) {};
				\dynkin[at=(current),name=lower]A3
				\draw (upper root 3.center) -- (lower root 2.center);
			\end{dynkinDiagram} & $E_8$ & 
			\begin{dynkinDiagram}[name=upper]A4
				\node (current) at ($(upper root 1)+(0,-.3cm)$) {};
				\dynkin[at=(current),name=lower]A4
				\draw (upper root 4.center) -- (lower root 3.center);
			\end{dynkinDiagram} \\[.1cm]
			&\begin{tikzpicture}
				\draw[->] (0,.5) edge (0,0) (.35cm, .5) edge (.35cm, 0) (.7cm, .5) -- (.7cm, 0);
			\end{tikzpicture} & &
			\begin{tikzpicture}
				\draw[->] (0,.5) edge (0,0) (.35cm, .5) edge (.35cm, 0) (.7cm, .5) edge (.7cm, 0) (1.05cm, .5) -- (1.05cm, 0);
			\end{tikzpicture} \\
			$H_3$ & \dynkin[backwards]{H}{3} & $H_4$ & \dynkin[backwards]{H}{4}   \\
			\hline
		\end{tabular}
		%\caption{The foldings $D_4 \to B_3,\ A_5 \to C_3,\ \tA_3 \to \tC_2,$ and $\tD_4 \to \tG_2$, respectively.}
		%\label{fig:Hfolding}
	\end{figure}
	
	\vspace{.2em}
	In~\cite[Proposition 1.4]{Dyer2009a} it is shown that there is a bijection $f:\uPhi \to \Phi\sqcup \tau\Phi$, where $\tau$ is the golden ratio $\frac{1+\sqrt{5}}{2}$. This bijection is the restriction of a linear map $p:\uV\to V$, and sends $\uPi$ to $\Pi\sqcup\tau\Pi$. Let $\alpha$ be a root of $\Phi$. If $f(\ualpha)=\alpha$ and $f(\ualpha')=\tau\alpha$, then we define 
	\[i(\alpha)\coloneqq \ualpha+\tau\ualpha'.\]
	% Remark: the more natural normalization of the map in the context of the paper is
	% \alpha \mapsto \frac{1}{\sqrt{5}}(\ualpha'-\overline{\tau}\ualpha')
	Then $i$ extends to a linear map $V\to\uV$. We thus have a system $V \overset{i}{\longrightarrow} \uV  \overset{p}{\longrightarrow}  V$ such that:
	
	\vspace{.5em}\noindent
	(\textbf{Condition $\mathbf{1'}$}) For every $\widehat{\alpha}_i \in \widehat{\Pi}$, there is a simple root $\alpha_j$ of $\Pi$ such that $p(\widehat{\alpha}_i) \in \RR_{>0} \alpha_j$.\\
	(\textbf{Condition $\mathbf{2'}$}) For every $\widehat{\beta} \in \uPhi$, there is a root $\beta$ of $\Phi$ such that $p(\widehat{\beta}) \in \RR_{>0} \beta$.\\
	(\textbf{Condition $\mathbf{3'}$}) For every $\beta \in \Phi$, the vector $i(\beta)$ is in $\Span_+ \{ \widehat{\beta} \in \uPhi : p(\widehat{\beta}) \in \RR_{>0} \beta \}$. 
	
	\vspace{.5em}
	With minimal changes, the proofs of \Cref{lem:PullbackOrderIdeal,lem:PullbackClosed,lem:PushforwardSeparable} work using these primed Conditions and the root preorder on $H_3$ and $H_4$ in place of the conditions in \Cref{sec:Folding}. The argument that $I$ is clean then proceeds as it does for types $B_3$ and $C_3$ in \Cref{sec:FoldedTypes}. We detail the adjusted proof of \Cref{lem:PullbackOrderIdeal}; the rest are similar.
	
	Let $\widehat{I}$ be the set of roots $\widehat{\beta}$ in $\uPhi^+$ such that $p(\widehat{\beta})$ is a positive multiple of some $\beta \in I$. 
	We claim that $\widehat{I}$ is an order ideal of $\uPhi^+$. Indeed, let $\widehat{\gamma} \in \widehat{I}$ and $\widehat{\beta} \in \uPhi$ with $\ugamma \succeq \ubeta$, we must verify that $\ubeta \in \widehat{I}$.
	Let $p(\widehat{\beta}) = x \beta$ and $p(\widehat{\gamma}) = y \gamma$, for $\beta$, $\gamma \in \pos$ and $x$, $y \in \RR_{>0}$. 
	We know that $\widehat{\gamma}-\widehat{\beta}$ is in the positive span of the $\widehat{\alpha}_i$'s, and thus $y \gamma - x \beta$ is in $\Span_+(\alpha_i)$, and thus $\gamma - (x/y) \beta$  is in $\Span_+(\alpha_i)$.
	This shows that $\gamma \succeq \beta$. It follows that $\beta$ is in $I$ and therefore $\ubeta$ is in $\widehat{I}$.
	%Now, let $B$ be biclosed in $I$ and let $\widehat{B}$ be the set of roots $\widehat{\beta}$ in $\uPhi^+$ such that $p(\widehat{\beta})$ is a positive multiple of some $\beta \in B$. 
	%Then $\widehat{B}$ is biclosed in $\widehat{I}$. Since we have proved that order ideals in $D_6$ and $E_8$ are clean, there is some $\widehat{\theta}$ in $\widehat{V}^{\ast}$ separating $\widehat{B}$ inside $\widehat{I}$.
	%Then $i^{\ast}(\widehat{\theta})$ separates $B$ inside $I$.
\end{proof}

%I have done this; glad to show you the Mathematica code, or let you replicate in SAGE.
%Here is the point. Let R={beta1, beta2, beta3, beta4, beta5} be a rank 2 subsystem in Phi^+ with 5 positive roots, and fundamental roots beta1 and beta5.
%The preimage of R in uPhi^+ is an A4 subsystem. If I contains some, but not all, of {beta2, beta3, beta4}, then the preimage in uPhi is not an order ideal. 
%Thus, beta2, beta3 and beta4 must be equivalent for the preorder on Phi^+.
%I wrote a little Mathematica notebook which ran through all such parabolics and took the transitive closure of this relation, and found that the equivalence classes match the boring ones here. 
%It is then easy to check that the preorder must match as well.
The authors have verified that there is no weaker preorder on the $H_3$ or $H_4$ root systems such that, if $I$ is an order ideal for that preorder, then $\widehat{I}$ is an order ideal in $\uPhi^+$.
We have not investigated whether there might be other orderings, not coming from folding, which are suitable.

%\section{Conjectures on hyperbolic rank $3$ root systems}

\bibliographystyle{amsplain}
\bibliography{BiclosedLattice}

\end{document}